\documentclass[12pt]{article}

\usepackage{graphicx}
\usepackage{latexsym,amssymb}
\usepackage{amsthm}
\usepackage{indentfirst}
\usepackage{amsmath}
\usepackage{color}
\usepackage{xcolor}
\usepackage{fourier}
\usepackage[colorlinks=true,backref=page]{hyperref}
\usepackage[all]{xy}
\usepackage{float}

\newtheorem{theoremalph}{Theorem}

\newtheorem*{Main Theorem}{Main Theorem}

\newtheorem{Theorem}{Theorem}[section]
\newtheorem*{Theorem A}{Theorem A}
\newtheorem*{Theorem A'}{Theorem A'}
\newtheorem*{Theorem B'}{Theorem B'}

\newtheorem{Proposition}[Theorem]{Proposition}
\newtheorem{Lemma}[Theorem]{Lemma}

\newtheorem{Remark}[Theorem]{Remark}

\newtheorem{Remark-numbered}[Theorem]{Remark}

\newtheorem*{Claim}{Claim}
\newtheorem{Claim-numbered}[Theorem]{Claim}
\newtheorem*{Acknowledgements}{Acknowledgements}

 \def\NN{{\mathbb N}}

 \def\ZZ{{\mathbb Z}}

   \def\cR{{\cal R}}

\def\dim{\operatorname{dim}}

\def\diam{\operatorname{Diam}}

\begin{document}

\title{Measures of maximal entropy for $C^\infty$ three-dimensional flows}

\author{Yuntao Zang\footnote{
Yuntao Zang is partially supported by National Key R\&D Program of China (2022YFA1005802) and NSFC (12201445, 12471186).
}}

\maketitle

\begin{abstract} 

\medskip

We prove that every $C^\infty$ non-singular flow with positive entropy on a compact three-dimensional manifold without boundary admits finitely many ergodic measures of maximal entropy. This result extends the notable work of Buzzi-Crovisier-Sarig (\emph{Ann. of Math.}, 2022) on surface diffeomorphisms. Our approach differs by addressing the continuity of Lyapunov exponents and the uniform largeness of Pesin sets for measures of maximal entropy. Furthermore, it provides an alternative proof for the case of surface diffeomorphisms.	

\medskip
\noindent\textbf{\normalsize MSC Class:} 37A35; 37B40; 37C10; 37C40; 37D25
\end{abstract}
\tableofcontents

\section{Introduction}
\subsection{Main results}
Let $X$ be a vector field over a compact Riemannian manifold $M$ without boundary and let $\varphi=\{\varphi^{t}\}$ be the flow generated by $X$. The variational principle in ergodic theory relates the topological entropy (denoted by $h_{\rm top}(X)$ or $h_{\rm top}(\varphi)$) and the measure-theoretic entropy (denoted by $h(X,\mu)$ or $h(\varphi,\mu)$ for an invariant measure $\mu$): $$h_{\rm top}(X)=\sup_{\mu: \text{ invariant}}h(X,\mu)=\sup_{\mu: \text{ ergodic}}h(X,\mu).$$ An invariant measure $\mu$ is said to be \emph{of maximal entropy} if $h(X,\mu)=h_{\rm top}(X)$, i.e., $\mu$ realizes the supremum in the above variational principle. We often use the abbreviation "MME" for "measure of maximal entropy" throughout this paper. For general dynamical systems, there might be no ergodic MMEs or infinitely many ergodic MMEs. So it is natural to ask whether a given system admits finitely many ergodic MMEs. This was first established for uniformly hyperbolic systems, such as Axiom A diffeomorphisms \cite{Bow71} and flows \cite{Bow74}, subshifts of finite type \cite{Par64}. Beyond uniformly hyperbolic systems, it was  later extended to non-uniformly expanding maps and non-uniformly hyperbolic diffeomorphisms including topologically transitive Markov shifts with countable states \cite{Gur70}, smooth interval maps \cite{Buz97}, systems derived from Anosov \cite{BuF13}, \cite{CFT18}, \cite{Ure12}, standard maps \cite{Oba21}, partially hyperbolic diffeomorphisms \cite{LOP26}, \cite{MoP24}, \cite{HHT12}, surface diffeomorphisms with large entropy \cite{BCS22'}, \cite{BLY24} and surface geodesic flows \cite{BCT18}, \cite{CKW21}, \cite{LOP25}.

In this paper, we prove:

\begin{theoremalph}\label{Thm:finiteness C infinity}
Let $X$ be a $C^\infty$ non-singular vector field with $h_{\rm top}(X)>0$ over a three-dimensional compact manifold $M$ without boundary. Then $X$ admits finitely many ergodic measures of maximal entropy.
\end{theoremalph}
\begin{Remark}\noindent
	\begin{enumerate}
		\item This can be viewed as a flow version of Buzzi-Crovisier-Sarig's notable work \cite{BCS22'} on the finiteness of MMEs  for $C^{\infty}$ surface diffeomorphisms. However, our approach differs substantially, as it is based on two other influential works by Buzzi-Crovisier-Sarig \cite{BCS22} and Burguet \cite{Bur24} concerning the continuity of Lyapunov exponents. Our approach also provides an alternative proof of the finiteness of ergodic MMEs for $C^{\infty}$ surface diffeomorphisms.
		\item The condition \emph{non-singular} can be somewhat relaxed in a certain sense. In fact, it is only required to ensure that the Oseledets sub-bundles corresponding to the positive and negative Lyapunov exponents of a MME are one-dimensional. In other words, the entropy is concentrated along curves (the stable/unstable manifolds). The assumption that $X$ is non-singular and three-dimensional arises for the same reason.
		\item Our work does not allow us to establish analogous results to those obtained by Buzzi-Crovisier-Sarig \cite{BCS22'} for surface diffeomorphisms-namely, the uniqueness of MME for topologically transitive systems and the Bernoulli property of this measure in the topologically mixing case \footnote{It is worth noting that Ledrappier-Lima-Sarig \cite{LLS16} proved that for three-dimensional non-singular flows with positive entropy, the flow with respect to the ergodic MME is either Bernoulli, or isomorphic to the product of a Bernoulli flow and a rotational flow.}. These conclusions in their paper rely heavily on the two-dimensional setting, exploiting geometric and spectral structures specific to surface dynamics.
		\item It is worth noting that Theorem~\ref{Thm:finiteness C infinity} has recently been extended to equilibrium measures for some potentials (Burguet-Luo-Yang \cite{BLY25}) and to singular three-dimensional flows (Li-Liu \cite{LiL26}, Buzzi-Crovisier-Lima-Yang-Luo \cite{BCL26}).
		
	\end{enumerate}
	
\end{Remark}
Indeed, we will prove a finite smooth version of Theorem \ref{Thm:finiteness C infinity} which is more general. In this paper, whenever we consider a $C^r$ ($r > 1$) diffeomorphism or vector field (flow), we allow $r$ to be a non-integer (i.e., a real number).

Recall that $\varphi=\{\varphi^{t}\}$ is the flow generated by $X$. We define the asymptotic dilation of $X$, $$\lambda^{+}(X):=\lim_{n\to+\infty}\frac{1}{n}\log\left\|D\varphi^{n}\right\|, \quad \lambda^{-}(X):=\lim_{n\to+\infty}\frac{1}{n}\log\left\|D\varphi^{-n}\right\|$$ where $\left\|D\varphi^{n}\right\|:=\sup_{x}\left\|D\varphi^{n}_{x}\right\|\geq 1.$
\begin{theoremalph}\label{Thm:finiteness C r}
	Let $X$ be a $C^r(r>1)$ non-singular vector field with $h_{\rm top}(X)>\frac{1}{r}\min\left\{\lambda^{+}(X), \lambda^{-}(X)\right\}$ over a three-dimensional compact manifold $M$ without boundary. Then $X$ admits finitely many ergodic measures of maximal entropy.
\end{theoremalph}

\begin{Remark}\noindent
	\begin{enumerate}
		\item We note that if we assume $X$ is $C^\infty$, the term $\frac{1}{r}\min\left\{\lambda^{+}(X), \lambda^{-}(X)\right\}$ can be replaced by zero, making Theorem \ref{Thm:finiteness C infinity} a direct consequence of Theorem \ref{Thm:finiteness C r}.
		\item The threshold $\frac{1}{r}\min\left\{\lambda^{+}(X), \lambda^{-}(X)\right\}$ may be sharp for the finiteness of ergodic MMEs in some sense (see \cite[Conjecture 1]{BCS22'} and \cite{Buz14}) \footnote{
			In \cite{Buz14}, Buzzi constructed a $C^{r}$ ($1\le r<\infty$) surface diffeomorphism with no measure of maximal entropy. The corresponding suspension flow with constant roof function $1$ therefore yields a three-dimensional $C^{r}$ flow with no measure of maximal entropy. On the other hand, \cite[Conjecture~1]{BCS22'} predicts that for any $\varepsilon>0$ and $1<r<\infty$, there exists a $C^{r}$ surface diffeomorphism $f$ with entropy larger than $\frac{1}{r}\min\left\{\lambda^{+}(f), \lambda^{-}(f)\right\}-\varepsilon$ and infinitely many ergodic measures of maximal entropy.}.  Our term $\frac{1}{r}\min\left\{\lambda^{+}(X), \lambda^{-}(X)\right\}$ is inspired by the work of Buzzi-Crovisier-Sarig \cite{BCS22} (where the corresponding threshold is $\frac{\lambda^{+}(f)+\lambda^{-}(f)}{r-1}$) and Burguet \cite{Bur24} (where the corresponding threshold is $\frac{\lambda^{+}(f)}{r}$), both of which analyze the continuity of the Lyapunov exponents (see Theorem \ref{invariant Cr case}) for a $C^{r}$ surface diffeomorphism $f$. To prove Theorem \ref{Thm:finiteness C r}, we primarily adopt Burguet's approach, as it yields a slightly sharper threshold.
		\item In Buzzi-Crovisier-Sarig's work \cite[Section 1.3]{BCS22'}, a more general result for measures with large entropy (not necessarily MME) is presented. To be precise, for a $C^{r}$ surface diffeomorphism $f$, the number of homoclinic classes carrying ergodic measures with entropy strictly larger than $\frac{\lambda^{+}(f)}{r}$ is finite. However, we are unable to establish an analogous result for flows, primarily due to the absence of results on the continuity of Lyapunov exponents in finite smoothness for measures with large entropy, which forms the basis of our approach. 
		\item Our proof of Theorem \ref{Thm:finiteness C r} also yields an alternative proof in the case of surface diffeomorphisms, which corresponds to the main result of Buzzi-Crovisier-Sarig in \cite{BCS22'}.
	\end{enumerate}
	
\end{Remark}

Let us remark on Buzzi-Crovisier-Sarig's recent work \cite{BCS22'} on the finiteness of MMEs for $C^{\infty}$ surface diffeomorphisms with positive entropy. There are three steps in their approach (see \cite[Section 1.10]{BCS22'}) to prove the finiteness of MMEs (an additional subsequent step proves the uniqueness of MME assuming transitivity). The first step is to associate each ergodic hyperbolic measure with a homoclinic class (two ergodic hyperbolic measures share the same homoclinic class if they are homoclinically related) which is based on Pesin theory and is not restricted to two-dimensional systems. The second step is to prove that there are at most finitely many such homoclinic classes. In this step which plays a key role, several strongly 2-dimensional arguments  are used. They introduced and heavily used a special kind of topological disks called \emph{$su$-quadrilaterals} to show the transverse intersections between stable and unstable manifolds (which gives that two "nearby" ergodic hyperbolic measures with large entropy must be homoclinically related and share the same homoclinic class). The third step is to code each homoclinic class in a finite-to-one way by an irreducible Markov shift with countable states. The irreducibility is the main novelty which improves Sarig's coding in \cite{Sar13}. By the work of Gurevich \cite{Gur70}, irreducible Markov shift with countable states can have at most one MME and consequently, each homoclinic class can carry at most one MME which proves the finiteness of MMEs. We note that the main techniques in the third step are generalized to the three-dimensional case for non-singular flows by Buzzi-Crovisier-Lima \cite{BCL25}.

For $C^{\infty}$ three-dimensional non-singular flows, to prove the finiteness of MMEs, the main difficulty arises from the second step mentioned above where the concept of $su$-quadrilaterals is quite restricted to the two-dimensional case. Inspired by the work of Buzzi-Crovisier-Sarig \cite{BCS22} and Burguet \cite{Bur24} on the continuity of Lyapunov exponents, we develop a new approach to overcome this difficulty without relying on a substitute for $su$-quadrilaterals. More precisely, by using the convergence of the stable/unstable lifts of MMEs, we can directly show that for MMEs, the geometry of the stable and unstable manifolds of points in a set with positive measure is uniformly controlled. Consequently, any two nearby MMEs must be homoclinically related which leads to the finiteness of MMEs. Indeed, our approach is natural to be expected: since the Lyapunov exponents of MMEs are uniformly bounded away from zero (except the flow direction which always carries zero exponent), MMEs should have uniformly large measure on a Pesin set where the hyperbolicity is well controlled. In some sense, this very general idea is already mentioned in Buzzi-Crovisier-Sarig's work \cite[Page 433]{BCS22'}, but they commented that they were not able to follow this approach. The main problem is that even though the Lyapunov exponents of MMEs are uniformly bounded away from zero, the time required to reach these exponents may not be uniformly bounded, potentially resulting in the undesirable case where nearby MMEs do not share the same Pesin set. Following two notable works of Buzzi-Crovisier-Sarig \cite{BCS22} and Burguet \cite{Bur24} on the continuity of the Lyapunov exponents, we show this cannot occur for $C^{\infty}$ systems. In other words, nearby MMEs must have uniformly large measure on a common Pesin set. We also note that if we assume some partial hyperbolicity (as considered in a recent work by Mongez-Pacifico \cite{MoP24} where they assume unstable entropy exceeds the stable entropy), the aforementioned problem can be directly resolved (See Remark 4.1 in \cite{MoP24}). 

Our approach also has certain limitations. Since we do not thoroughly address the homoclinic class, we do not obtain a spectral decomposition as in Buzzi-Crovisier-Sarig's work (see \cite[Theorem 1]{BCS22'}), nor do we provide a description of the dynamics within each homoclinic class (see \cite[Theorem 2]{BCS22'}). As a result, we are unable to establish the uniqueness of MME under the assumption of transitivity. It is worth noting, however, the recent works of Lima-Obata-Poletti on the finiteness and uniqueness of MMEs for certain non-uniformly hyperbolic systems \cite{LOP26} and for geodesic flows on surfaces \cite{LOP25}.

We also highlight a recent work by Buzzi-Crovisier-Sarig \cite{BCS25}, in which they introduced the notion of \emph{strong positive recurrence (SPR)}. They demonstrated that surface diffeomorphisms with large entropy satisfy the SPR property and established several significant consequences, such as the existence and finiteness of ergodic MMEs, exponential decay of correlations, and more. Our Theorem \ref{Pro:compare-two-sets} and Theorem \ref{Thm:limit measure carries large exponent everywhere} share some similarities with their SPR property. However, the main difference is that we allow the presence of a center direction (the flow direction), whereas their SPR property is defined only for a Pesin set with a hyperbolic splitting, i.e., $T_{x}M = E_{x}^{s} \oplus E_{x}^{u}$ (see Definition 1.2 in \cite{BCS25}).

\subsection{Outline of the proof of Theorem \ref{Thm:finiteness C r}}

\paragraph{\it{\bfseries Step 1. Convergence of stable/unstable lifts of MMEs}}

Let $\{\mu_{k}\}$ be a sequence of ergodic MMEs converging to some invariant measure $\mu$ (which is also a MME). We show, under the condition $h_{\rm top}(X)>\frac{1}{r}\min\left\{\lambda^{+}(X), \lambda^{-}(X)\right\}$, that $$ \hat\mu_k^-\to\hat\mu^-,\,\,\hat\mu_k^+\to\hat\mu^+$$ where $\hat{\mu}_k^{-/+}, \hat{\mu}^{-/+}$ are the corresponding stable/unstable lifts (they are well-defined, see Section \ref{tangent dynamics}). The key idea is primarily derived from the work of Burguet \cite{Bur24} (a strengthened extension of Buzzi-Crovisier-Sarig's  results \cite{BCS22}) on the continuity of Lyapunov exponents, although they do not explicitly frame it in this manner. We remark that the idea of Buzzi-Crovisier-Sarig also serves our purpose but may require a slightly stronger condition, namely, $$h_{\rm top}(X) > \frac{\lambda^{+}(X) + \lambda^{-}(X)}{r-1}.$$

\paragraph{\it{\bfseries Step 2. Uniform largeness of the Pesin sets}}
Building on step 1, we further deduce that nearby MMEs have uniformly positive measure on a common Pesin set and hence must be homoclinically related. This conclusion is derived by combining Theorem \ref{Pro:compare-two-sets} and Theorem \ref{Thm:limit measure carries large exponent everywhere}.

\paragraph{\it{\bfseries Step 3. Homoclinic class and finiteness of MMEs}}
A recent work of Buzzi-Crovisier-Lima (\cite[Corollary 1.2]{BCL25}) demonstrates that two homoclinically related MMEs must coincide, which in turn implies the finiteness of MMEs.

\section{Preliminaries}
Let $X$ be a non-singular vector field over a three-dimensional compact manifold $M$ without boundary and let $\varphi=\{\varphi^{t}\}$ be the flow generated by $X$. Let $\varphi^{1}$ be the time-one map of $\varphi$. For simplicity, we will often consider the time-one map instead of the flow, as entropies, Lyapunov exponents, and many other invariants can be defined for the time-one map and coincide with those of the flow.
\subsection{Entropy and Lyapunov exponents}
The topological entropy of $X$ is denoted by $h_{\rm top}(X)$ or $h_{\rm top}(\varphi)$. Let $h(\varphi, \mu)$ or $h(X, \mu)$ denote the metric entropy of an invariant measure $\mu$. By Oseledets Multiplicative Ergodic Theorem, there is an invariant set $\mathcal{R}$ (called the \emph{Lyapunov regular set}) with total measure (i.e., $\mu(\mathcal{R})=1$ for any invariant measure $\mu$) such that for any $x\in\mathcal{R}$, there are a splitting (called the  \emph{Oseledets splitting}) $T_{x}M=E^{1}\oplus E^{2}\oplus\cdots\oplus E^{l}$ and finitely many numbers (called the \emph{Lyapunov exponents}) $\lambda_{1}(x)>\lambda_{2}(x)>\cdots>\lambda_{l}(x)$ such that for any nonzero vector $v\in E^{j}$, we have $$\lim_{t\to \pm\infty}\frac{1}{t}\log\left\|D\varphi^{t}_{x}(v)\right\|=\lambda_{j}(x) $$ and $$\lim_{t\to \pm\infty}\frac{1}{t}\log\measuredangle\left( D\varphi^{t}_{x}(E^{i}_{x}), D\varphi^{t}_{x}(E^{j}_{x})\right)=0,\quad i\neq j.$$

In our special setting, there are at most three Lyapunov exponents, one of which is zero, corresponding to the flow direction. If the largest Lyapunov exponent $\lambda_{1}(x)$ is positive, we denote it by $\lambda^{+}(\varphi,x)$ and it can also be defined as $$\lambda^{+}(\varphi,x)=\lim_{t\to+\infty}\frac{1}{t}\log\left\|D\varphi^{t}_{x}\right\|.$$ Similarly, if the smallest Lyapunov exponent is negative, we denote it by $\lambda^{-}(\varphi,x)$. In this case, we rewrite the Oseledets splitting at $x$ as $$T_{x}M=E^{+}_{x}\oplus E_{x}^{0}\oplus E_{x}^{-}$$ where $E_{x}^{0}={\rm Span}(X(x))$ is the flow direction. An invariant measure $\mu$ is called \emph{hyperbolic} if for $\mu$-a.e. $x$, the three Lyapunov exponents of $x$ are positive, zero (flow direction) and negative. For an ergodic measure $\mu$ with $h(\varphi,\mu)>0$, $\mu$ is hyperbolic by Ruelle's inequality.  For a hyperbolic measure $\mu$, we define $$\lambda^{+}(\varphi, \mu):=\int\lambda^{+}(\varphi,x) d\,\mu(x),\quad \lambda^{-}(\varphi, \mu):=\int\lambda^{-}(\varphi,x) d\,\mu(x).$$ If $\mu$ is ergodic, then $\lambda^{\pm}(\varphi, \mu)=\lambda^{\pm}(\varphi,x)$ for $\mu$-a.e. $x$. For a diffeomorphism $f$ on $M$, we can similarly define these objects.

\subsection{Tangent dynamics}\label{tangent dynamics}
With the Riemannian structure inherited from the tangent bundle $TM$, we denote the projective tangent bundle of $M$ by  $$\hat{M}:=\left\{(x, E):~x\in M, E \text{ is a one-dimensional linear subspace of $T_{x}M$}\right\}.$$ Let $\pi:\hat{M}\to M$ be the natural projection. Let $\hat{\varphi}$ be the induced flow of $\varphi$ (also called the \emph{canonical lift}) on $\hat{M}$ defined by $$\hat{\varphi}^{t}(x, E):=\left(\varphi^{t}(x), D\varphi^{t}_{x}(E)\right).$$ If $\varphi$ is of class $C^{r}$, then $\hat{\varphi}$ is of class $C^{r-1}$. Let $\rho:\hat{M}\to\mathbb{R}$ be the continuous function $$\rho(x, E):=\log\left\|D\varphi^{1}_{x}|_{E}\right\|$$ where $\varphi^{1}$ is the time-one map of $\varphi$. Note that both $\hat{\varphi}$ and $\rho$ are well-defined, since they depend only on the one-dimensional subspace $E$ and not on any particular choice of a nonzero vector spanning $E$.

Given a $\hat{\varphi}$-invariant measure $\hat{\mu}$, the \emph{projection} of $\hat{\mu}$ is defined by $\mu:=\hat{\mu}\circ\pi^{-1}$ and $\hat{\mu}$ is called the \emph{lift} of $\mu$. We define $$\lambda(\hat{\varphi},\hat{\mu}):=\int_{\hat{M}}\rho d\,\hat{\mu}.$$ We list some basic properties. 

\begin{itemize}
	\item If $\hat{\mu}_{n}\xrightarrow{\text{weak $\ast$}}\hat{\mu}$, then $\lambda(\hat{\varphi},\hat{\mu}_{n})\to \lambda(\hat{\varphi},\hat{\mu})$. This is a consequence of the continuity of the function $\rho$.
	\item Let $\mu$ be a hyperbolic measure. We can define the two lifts: $$\hat{\mu}^{+}:=\int_{\hat{M}}\delta_{(x, E^{+}_{x})} d\,\mu(x),\quad\hat{\mu}^{-}:=\int_{\hat{M}}\delta_{(x, E^{-}_{x})} d\,\mu(x).$$ $\hat{\mu}^{+}, \hat{\mu}^{-}$ are called the \emph{unstable and stable lifts} of $\mu$. Moreover, we have $$\lambda^{\pm}(\varphi, \mu)=\lambda(\hat{\varphi},\hat{\mu}^{\pm}).$$  If $\mu$ is $\varphi$-ergodic, these two lifts are $\hat{\varphi}$-ergodic. Moreover (refer to \cite[Proposition 3.8]{BCS25}), $$\lambda^{\pm}(\varphi,\mu)=\lambda(\hat{\varphi},\hat{\mu})\quad\text{if and only if}\quad\hat{\mu}=\hat{\mu}^{\pm}.$$ 
	
	For detailed discussions, we refer to \cite[Lemma 3.3]{BCS22} and the references therein.
	\item Let $\mu$ be a hyperbolic measure. Then any lift $\hat{\mu}$ of $\mu$ is carried by $\left\{(x, E^{-}_{x})\right\}, \left\{(x, E^{0}_{x})\right\}$ and $\left\{(x, E^{+}_{x})\right\}$ (refer to \cite[Corollary 3.4]{BCS22} for detailed arguments).

	\item Let $\hat{\mu}$ be an invariant measure and let $\mu$ be its projection. We have $h(\hat{\varphi},\hat{\mu})=h(\varphi,\mu)$ as an application of Ledrappier-Walters variational principle \cite{LeW77} and the work of Sacksteder-Shub (\cite[Lemma 3.3]{SaS78}\footnote{In fact, this remains true in arbitrary dimensions. Indeed, \cite[Lemma~3.3]{SaS78} essentially proves that for any diffeomorphism $f$ on a compact manifold $M$ of arbitrary dimension,
		$$
		h_{\rm top}\left(\hat{f},\pi^{-1}(x)\right)=0,\qquad \forall x\in M.
		$$
		This observation was noted by Burguet in \cite[Section~5.3]{Bur24'}. It was also proved in detail for dimension two by Buzzi-Crovisier-Sarig \cite[Corollary 4.2]{BCS22}.}.
\end{itemize}
Similarly, for diffeomorphisms on $M$, these objects can be defined, and the corresponding properties remain valid.

\section{Convergence of stable/unstable lifts of MMEs}
In this section, the main result is the following theorem which shows that the convergence of MMEs implies the convergence of their stable/unstable lifts. 

Recall that the asymptotic dilation of $X$ is defined by $$\lambda^{+}(X):=\lim_{n\to+\infty}\frac{1}{n}\log\left\|D\varphi^{n}\right\|, \quad \lambda^{-}(X):=\lim_{n\to+\infty}\frac{1}{n}\log\left\|D\varphi^{-n}\right\|$$ where $\left\|D\varphi^{n}\right\|:=\sup_{x}\left\|D\varphi^{n}_{x}\right\|\geq 1.$
\begin{Theorem}\label{continuity of exponents for flow}
	Let $X$ be a $C^r(r>1)$ non-singular vector field with $h_{\rm top}(X)>\frac{1}{r}\min\left\{\lambda^{+}(X), \lambda^{-}(X)\right\}$ over a three-dimensional compact manifold $M$ without boundary. If $\{\mu_k\}$ is a sequence of $\varphi^{1}$-ergodic measures with $\mu_k\to\mu$ (for some $\varphi^{1}$-invariant measure $\mu$) and $h(X, \mu_{k})\to h_{\rm top}(X)$, then $\mu$ is a MME and the corresponding unstable and stable lifts converge, i.e., $$\lim_{k\to+\infty}\hat\mu_k^+=\hat\mu^+, \quad \lim_{k\to+\infty}\hat\mu_k^-=\hat\mu^-.$$
\end{Theorem}

\begin{Remark}\noindent
	\begin{enumerate}
		\item Similar results were established for surface diffeomorphisms by Buzzi-Crovisier-Sarig \cite{BCS22} and Burguet \cite{Bur24},  although they do not explicitly frame it in this manner. Indeed, their focus is on establishing the continuity of the Lyapunov exponents: $$\lambda^{+}(f, \mu_{k})\to\lambda^{+}(f, \mu),\quad \lambda^{-}(f, \mu_{k})\to\lambda^{-}(f, \mu)$$ which follows as a consequence of the convergence of the unstable/stable lifts in Theorem \ref{continuity of exponents for flow}. We note that, for $C^{\infty}$ surface diffeomorphisms, establishing the continuity of the Lyapunov exponents does not necessarily require assuming that $h(f, \mu_k) \to h_{\rm top}(f)$ in some cases. For instance, one can instead assume that $h(f, \mu_k) \to h(f, \mu) > 0$ and $\mu$ is ergodic (see the proof of Theorem \ref{continuity of exponents for flow} below and the subsequent remarks). Refer to \cite[Theorem A]{BCS22} for precise statements.
		
		The fact that the limit measure $\mu$ is also a MME is known for $C^{r}$ surface diffeomorphisms (see Burguet \cite[Corollary 1]{Bur24}). For $C^{\infty}$ systems in any dimension, refer to Newhouse \cite{New89} and Yomdin \cite{Yom87}.
		\item The threshold $\frac{1}{r}\min\left\{\lambda^{+}(X), \lambda^{-}(X)\right\}$ comes from Burguet's analysis of the continuity of Lyapunov exponents \cite{Bur24} (see Theorem \ref{invariant Cr case} below). The original idea comes from Buzzi-Crovisier-Sarig \cite{BCS22} for a surface diffeomorphism $f$ where the threshold is $\frac{\lambda^{+}(f) + \lambda^{-}(f)}{r-1}$ which is slightly larger than  $\frac{1}{r}\min\left\{\lambda^{+}(f), \lambda^{-}(f)\right\}$.
		\item A parallel result can be formulated for diffeomorphisms on three-dimensional manifolds, provided that every ergodic measure with large entropy has exactly one zero Lyapunov exponent (which implies they have exactly one positive and one negative exponent). See Theorem \ref{invariant Cr case} below for the key role of this condition.
	\end{enumerate}
\end{Remark}

\subsection{Continuity of the Lyapunov exponents: Proof of Theorem \ref{continuity of exponents for flow}}
To prove Theorem \ref{continuity of exponents for flow}, for simplicity, we primarily focus on the time-one map of the flow, which is a diffeomorphism. For a diffeomorphism $f$, we say an invariant measure $\mu$ has exactly one positive Lyapunov exponent if for $\mu$-a.e. $x$, we have $\lambda_{1}(x)=\lambda^{+}(f,x)>0$ and all other exponents are non-positive. Recall the asymptotic dilation $$\lambda^{+}(f):=\lim_{n\to+\infty}\frac{1}{n}\log\left\|Df^{n}\right\|, \quad \lambda^{-}(f):=\lim_{n\to+\infty}\frac{1}{n}\log\left\|Df^{-n}\right\|$$ where $\left\|Df^{n}\right\|:=\sup_{x}\left\|Df^{n}_{x}\right\|, n\in\mathbb{N}.$
\begin{Theorem}\label{invariant Cr case}
	Let $f$ be a $C^{r}(r>1)$ diffeomorphism on a compact manifold $M$ without boundary. Let $\{\nu_{k}\}_{k\geq 1}$ be a sequence of ergodic measures such that each $\nu_{k}$ has exactly one positive Lyapunov exponent and the unstable lifts $\hat{\nu}_{k}^{+}$ converge to some invariant measure $\hat{\mu}$ whose projection $\mu$ also has exactly one positive Lyapunov exponent.

	Then for any $\alpha>\frac{\lambda^{+}(f)}{r}$, there is a decomposition of $\hat{\mu}=(1-\beta)\hat{\mu}_{0}+\beta \hat{\mu}^{+}_{1}$ for some $\beta\in [0,1]$ and some invariant measures $\hat{\mu}_{0},\hat{\mu}_{1}^{+}$ ($\hat{\mu}_{1}^{+}$ is the unstable lift of some $f$-invariant measure $\mu_{1}$) such that $$\limsup_{k\to+\infty}h(f,\nu_{k})\leq\beta h(f,\mu_{1})+(1-\beta)\alpha.$$
\end{Theorem}
The proof of Theorem \ref{invariant Cr case} closely follows the argument presented by Burguet \cite{Bur24} for surface diffeomorphisms, which itself generalizes the work of Buzzi-Crovisier-Sarig \cite{BCS22}. While Burguet's methods are developed in the context of surface dynamics, in general, his conclusions can be naturally extended to three-dimensional flows. This extension involves certain adjustments to account for the differences between surface diffeomorphisms and three-dimensional flows. To maintain both clarity and rigor, and to ensure the results are presented in a self-contained manner, we include a summary of the proof of Theorem \ref{invariant Cr case} in Section \ref{outline of the proof of Theorem invariant Cr case }, and a detailed adaptation of Burguet's methods provided in Section \ref{Burguet's work}.
\begin{Remark}\noindent
	\begin{enumerate}
		\item The setting of this theorem is modeled on the three-dimensional non-singular flows, which are our primary focus. Although we do not assume that $h_{\rm top}(f) > \frac{\lambda^{+}(f)}{r}$, this is the only case of interest. If $h_{\rm top}(f) \leq \frac{\lambda^{+}(f)}{r}$, the theorem holds trivially by setting $ \hat{\mu}_{0} = \hat{\mu}$ and $\beta = 0$, yielding no meaningful conclusion.
		\item The theorem is essentially parallel with the work of Burguet \cite{Bur24} (which is a generalization of the work of Buzzi-Crovisier-Sarig \cite{BCS22}) for surface diffeomorphisms. The essential idea is that the entropy is concentrated on certain one-dimensional sub-manifolds (the unstable manifolds), a property that naturally arises with hyperbolic measures for surface diffeomorphisms. Given that we are working with three-dimensional non-singular flows, this approach is broadly applicable.
	\end{enumerate}
	
\end{Remark}

Theorem \ref{continuity of exponents for flow} is a consequence of Theorem \ref{invariant Cr case}.

\begin{proof}[Proof of Theorem \ref{continuity of exponents for flow} assuming Theorem \ref{invariant Cr case}]
	
	The idea is to apply Theorem \ref{invariant Cr case} to the time-one map (denoted by $f$ for convenience). Without loss of generality, we may assume $h_{\rm top}(f)>\frac{\lambda^{+}(f)}{r}$ (otherwise consider $f^{-1}$ instead).
	
	Assume the unstable lifts $\hat{\mu}_{k}^{+}\to \hat{\mu}$ for some invariant measure $\hat{\mu}$ on $\hat{M}$ whose projection is $\mu$. Let $\alpha$ be such that $h_{\rm top}(f)>\alpha>\frac{\lambda^{+}(f)}{r}$. Applying Theorem \ref{invariant Cr case} to the time-one map $f$, we have a decomposition of $\hat{\mu}=(1-\beta)\hat{\mu}_{0}+\beta \hat{\mu}_{1}^{+}$ for some $\beta\in [0,1]$ and invariant measures  $\hat{\mu}_{0}, \hat{\mu}_{1}^{+}$ ($\hat{\mu}_{1}^{+}$ is the unstable lift of some hyperbolic measure $\mu_{1}$) such that $$h_{\rm top}(f)=\lim_{k\to+\infty}h(f,\mu_{k})\leq\beta h(f,\mu_{1})+(1-\beta)\alpha\leq \beta h_{\rm top}(f)+(1-\beta)\alpha.$$ Here, to save notation, we readopt the notation $\mu_{1}$ which should not be confused with the first element in the sequence $\{\mu_{k}\}$. The above inequality implies $\beta=1, \mu=\mu_{1}$ and consequently $\mu$ is a MME. Moreover, $\hat{\mu}=\hat{\mu}^{+}$. Note that $\hat{\mu}^{+}$ is the unique unstable lift of $\mu$ (since $\mu$ is assumed to have exactly one positive exponent). We then conclude that $\hat{\mu}_{k}^{+}\to \hat{\mu}=\hat{\mu}^{+}$.
	
	Next we use a trick borrowed from Buzzi-Crovisier-Sarig \cite[Lemma 3.10]{BCS25} to show that $\hat{\mu}_{k}^{+}\to \hat{\mu}^{+}$ actually implies $\hat{\mu}_{k}^{-}\to\hat{\mu}^{-}$ in our very special setting (i.e., the dynamics is the time-one map of a three-dimensional  non-singular vector field). By abusing notation, let us again assume the stable lifts $\hat{\mu}_{k}^{-}\to \hat{\mu}$ for some invariant measure $\hat{\mu}$ on $\hat{M}$. Our goal is to show $\hat{\mu}=\hat{\mu}^{-}$. By $\hat{\mu}_{k}^{+}\to \hat{\mu}^{+}$, we have the convergence of the positive Lyapunov exponents, i.e., 
	\begin{equation}\label{formula1 in Theorem 3.1}
		\lambda^{+}(f,\mu_{k})=\int\rho d\,\hat{\mu}_{k}^{+}\to \int\rho d\,\hat{\mu}^{+}=\lambda^{+}(f,\mu)
	\end{equation} where we recall that $\rho:\hat{M}\to\mathbb{R}$ is the continuous function $$\rho(x, E):=\log\left\|Df_{x}|_{E}\right\|.$$ In our special setting, the measures $\mu_{k}$  and $\mu$ have exactly three Lyapunov exponents: one positive, one zero (flow direction) and one negative. Note $$\lambda^{-}(f,\mu_{k})=\int\log \det \left(Df_{x}\right)d\mu_{k}(x)-\lambda^{+}(f,\mu_{k}).$$ Since $\log \det \left(Df_{x}\right)$ is a continuous function, by (\ref{formula1 in Theorem 3.1}), we have $$\lambda^{-}(f,\mu_{k})\to \lambda^{-}(f,\mu).$$ Since $\rho$ is a continuous function,  we have $$\lambda^{-}(f,\mu_{k})=\int\rho d\,\hat{\mu}_{k}^{-}\to \int\rho d\,\hat{\mu}.$$  Hence $$\lambda(\hat{f},\hat{\mu}):=\int\rho d\,\hat{\mu}=\lambda^{-}(f,\mu).$$ Therefore, we have (see the basic properties in Section \ref{tangent dynamics})  $$\hat{\mu}=\hat{\mu}^{-}.$$

\end{proof}
One can observe that the proof above can be adapted to some other setting: if $\mu_k$ is a sequence of ergodic hyperbolic measures (not necessarily MMEs) converging to an ergodic measure  $\mu$ with $\lim_{k\to+\infty}h(f,\mu_{k})= h(f,\mu)>\frac{1}{r}\min\left\{\lambda^{+}(f), \lambda^{-}(f)\right\}>0$ (ergodicity implies $\mu_{1}=\mu$ and $\beta=1$), then the corresponding unstable lifts also converge, i.e., $$\lim_{k\to+\infty}\hat\mu_k^+=\hat\mu^+, \quad \lim_{k\to+\infty}\hat\mu_k^-=\hat\mu^-.$$

\subsection{Heuristic overview of Theorem \ref{invariant Cr case}: summary of Burguet's work}\label{outline of the proof of Theorem invariant Cr case }
To present a heuristic outline of the proof of Theorem \ref{invariant Cr case}, we omit certain non-essential details and rely on approximate properties that may not hold in full generality. We provide a rigorous proof in Section \ref{Burguet's work}.

Let us disregard the role of the lifted dynamics and focus solely on proving the final conclusion:

 $$\limsup_{k\to+\infty}h(f,\nu_{k})\lesssim\beta h(f,\mu_{1})+(1-\beta)\cdot\frac{\lambda^{+}(f)}{r}.$$

\paragraph{\it{\bfseries Entropy splitting:}\\}
Given an ergodic measure $\nu$ with exactly one positive Lyapunov exponent, by Ledrappier-Young entropy theory, the entropy of $\nu$ coincides with the entropy of the conditional measure $\nu_{x}$ on the local unstable manifold $W^{u}_{\rm loc}(x)$ for a $\nu$-typical point $x$. To be precise, for a finite partition $\mathcal{P}$ with sufficiently small diameter and for sufficiently large $n$, $$h(f,\nu)\approx \frac{1}{n}H_{\nu_{x}}\left(\mathcal{P}^{[0,n)}\right)$$ where, given a subset $E\subset\mathbb{N}$, we define the partition $$\mathcal{P}^{E}:=\bigvee_{i\in E}f^{-i}\mathcal{P}.$$ 
We divide the interval $[0,n)$ into two parts $E$ and $[0,n)\setminus E$. The part $E$ can be viewed as the collection of "hyperbolic times" (see the beginning of Section \ref{Section Entropy estimation of neutral component} for a precise discussion) where uniform hyperbolicity is recovered. The entropy contribution is then separated into two components, referred to as the geometric and neutral components, $$H_{\nu_{x}}\left(\mathcal{P}^{[0,n)}\right)\lesssim H_{\nu_{x}}\left(\mathcal{P}^{E}\right)+ H_{\nu_{x}}\left(\mathcal{P}^{[0,n)\setminus E}\bigg|\mathcal{P}^{E}\right)=H_{\nu_{x}}\left(\mathcal{P}^{E}\right)+ H_{\nu_{x}}\left(\mathcal{P}^{[0,n)}\bigg|\mathcal{P}^{E}\right).$$

\paragraph{\it{\bfseries Geometric component:}\\}

Write $$\nu_{x}^{E}:=\frac{1}{\# E}\sum_{i\in E}f_{*}^{i}\nu_{x}.$$

A general result (see Lemma \ref{technical lemma}) provides a bound on the entropy of a partition over the set $E$ (i.e., $H_{\nu_{x}}(\mathcal{P}^{E})$) by relating it to the entropy of a measure averaged over the $E$-iterates of $f$ (i.e., $\nu^{E}_{x}$ defined above). Specifically, for any $m\in\mathbb{N}$, we have  $$\frac{1}{\# E}H_{\nu_{x}}\left(\mathcal{P}^{E}\right)\lesssim  \frac{1}{m}H_{\nu_{x}^{E}}\left(\mathcal{P}^{[0,m)}\right).$$ We write (assuming the limits exist): $$\beta_{\nu}:=\lim_{n\to+\infty}\frac{\# E}{n},\quad \nu^{*}:=\lim_{n\to+\infty} \nu_{x}^{E}.$$ Passing to the limit in $n$, we then have (assuming $\nu^{*}(\partial\mathcal{P}^{m})=0$) $$\lim_{n\to+\infty}\frac{1}{n}H_{\nu_{x}}\left(\mathcal{P}^{E}\right)\lesssim \beta_{\nu} \cdot \frac{1}{m}H_{\nu^{*}}\left(\mathcal{P}^{[0,m)}\right).$$

\paragraph{\it{\bfseries Neutral component:}\\}

We note
$$H_{\nu_{x}}\left(\mathcal{P}^{[0,n)}\bigg|\mathcal{P}^{E}\right)\lesssim \log\sup_{A\in \mathcal{P}^{E}}\#\left\{B\in\mathcal{P}^{[0,n)}:~ B\cap A\cap W^{u}_{\rm loc}(x)\neq \emptyset\right\}.$$

This inequality provides a rough estimate of how many elements $B \in \mathcal{P}^{[0,n)}$ are needed to cover $A \cap W^{u}_{\rm loc}(x)$, where $A \in \mathcal{P}^{E}$. In other words, it addresses the question: into how many segments must $A \cap W^{u}_{\rm loc}(x)$ be subdivided so that, for each resulting segment $\theta$ (which resembles an element $B \in \mathcal{P}^{[0,n)}$), the length of $f^{i}(\theta)$ remains below a fixed scale (say, 1) for all $i \in [0, n)$? This question aligns with a typical problem in Yomdin theory, often addressed through an inductive process.

At a "hyperbolic time" $k$ (i.e., $k \in E$), by the definition of $\mathcal{P}^{E}$, the size of $f^{k}(A)$ is less than 1. Consequently, for segments $\theta$ dividing $A \cap W^{u}_{\rm loc}(x)$ constructed during the preceding inductive step, the condition that the length of $f^{k}(\theta)$ is less than 1 is automatically satisfied, up to some negligible error.

However, at a non-hyperbolic time $k$ (i.e., $k \in [0, n) \setminus E$), the absence of "hyperbolicity" necessitates subdividing $\theta$ into smaller segments. The number of these subdivisions is bounded above by $e^{\frac{\lambda^{+}(f)}{r}}$, according to Yomdin theory. Since the number of non-hyperbolic times is approximately $n - \#E$, we obtain:  $$H_{\nu_{x}}\left(\mathcal{P}^{[0,n)}\bigg|\mathcal{P}^{E}\right)\lesssim \left(n-\# E\right)\cdot \frac{\lambda^{+}(f)}{r}.$$
Passing to the limit in $n$, we have $$\lim_{n\to+\infty}\frac{1}{n}H_{\nu_{x}}\left(\mathcal{P}^{[0,n)}\bigg|\mathcal{P}^{E}\right)\lesssim \left(1-\beta_{\nu}\right)\cdot \frac{\lambda^{+}(f)}{r}.$$
\paragraph{\it{\bfseries Applying to a sequence of ergodic measures $\nu_{k}$:}\\}
Above we have proved that for an ergodic measure $\nu$ with exactly one positive Lyapunov exponent, $$h(f,\nu)\lesssim \beta_{\nu} \cdot \frac{1}{m}H_{\nu^{*}}\left(\mathcal{P}^{[0,m)}\right)+\left(1-\beta_{\nu}\right)\cdot \frac{\lambda^{+}(f)}{r}.$$

We apply the estimations above to a sequence of ergodic measures $\nu_{k}$ assuming the corresponding $\nu_{k}^{*}\to \mu_{1}$ for some invariant probability measure $\mu_{1}$ (which is some component of the limit measure $\mu$) and the corresponding $\beta_{\nu_{k}}\to\beta$. By the arbitrariness of $m$ (assuming $\mu_{1}(\partial\mathcal{P}^{m})=0$), $$\limsup_{k\to+\infty}h(f,\nu_{k})\lesssim\beta \cdot \frac{1}{m}H_{\mu_{1}}\left(\mathcal{P}^{[0,m)}\right)+\left(1-\beta\right)\cdot \frac{\lambda^{+}(f)}{r}\lesssim\beta h(f,\mu_{1})+\left(1-\beta\right)\cdot \frac{\lambda^{+}(f)}{r}.$$

We remark that these measures $\nu_{k}^{*}$ and $\mu_{1}$ are actually not defined directly on the manifold $M$.  Instead, they are obtained as projections of the corresponding lifted measures under the tangent dynamics. For further details, see Section \ref{Burguet's work}.
\paragraph{\it{\bfseries Remarks on critical flaws:}\\}
We point out several problematic aspects of the above arguments and offer their corrections for careful readers.
\begin{itemize}
	\item Why is $\mu_1$ invariant under $f$, and why is its maximal Lyapunov exponent positive ?
	
	Indeed, it is usually \emph{NOT} $f$-invariant by our "naive" construction above. To address this, we can replace the set $E$ of "hyperbolic times" with the "$L$-expanded" set 
	$$
	E^{L}:=\bigcup [a,b),
	$$ 
	where the union runs over all $a,b\in E$ with $0<b-a\leq L$. In this new setup, the measure 
	$$
	\nu_{x}^{E^{L}} := \frac{1}{\# E^{L}} \sum_{i\in E^{L}} f_{*}^{i} \nu_{x}
	$$ 
	becomes "almost" $f$-invariant for sufficiently large $L$, since the boundary of the intervals in $E^{L}$ only occupies a small proportion of $E^{L}$. Taking the limit as $L \to +\infty$, the limit measure $\mu_{1}$ becomes $f$-invariant. The maximal Lyapunov exponent of $\mu_{1}$ is positive because it is mainly constructed from "hyperbolic times" and the hyperbolicity is inherited in some sense. The parameter $\alpha$ in Theorem \ref{invariant Cr case} serves as a technical adjustment that facilitates this result.
	
	\item Why do all the points in $W^{u}_{\rm loc}(x)$ share the same set $E$ of hyperbolic times ?
	
	Different points might have different hyperbolic times, but the total number of types of hyperbolic times grows only sub-exponentially in $L$, assuming we replace $E$ with $E^{L}$ as described above.  We can further divide $W^{u}_{\rm loc}(x)$ into sub-exponentially many segments, with all points in each segment sharing the same set of hyperbolic times. This sub-exponential division does not affect the entropy estimation. See the role of $\mathcal{E}_{n}^{L}$ in Proposition \ref{entropy splitting}.

	\item Why is it valid to use the same partition $\mathcal{P}$ to estimate the entropy, given that $\mathcal{P}$ fundamentally depends on the reference measures ?
	
	This issue is addressed by a technical adjustment. Specifically, two partitions are considered: one is dependent on the reference measure, while the other is independent. For further clarification, refer to Proposition \ref{entropy splitting} and the subsequent remark.
\end{itemize}

\section{Pesin theory}\label{Pesin theory}
\subsection{Size of the local stable and unstable manifolds}\label{size of local manifolds}
Recall that $\cR$ denotes the Lyapunov regular set. Given $\chi>0$, let ${\rm NUH}_{\chi}\subset\cR$ be the set of $x$ such that $x$ has one positive exponent which is strictly larger than $\chi$, one zero exponent which is on the flow direction and one negative exponent which is strictly less than $-\chi$. For $x\in {\rm NUH}_{\chi}$, recall that we denote the corresponding Oseledets splitting by $$T_{x}M=E^{+}_{x}\oplus E_{x}^{0}\oplus E_{x}^{-}$$ where $E_{x}^{0}={\rm Span}(X(x))$ is the flow direction. For these bundles, we drop the index $x$ if there is no confusion. 

Given $C>1$, we define $\cR_{C,\chi}\subset {\rm NUH}_{\chi}$ to be the set of $x$ such that for any $ n\in\NN$, 

\begin{itemize}
	\item  $\left\|D\varphi^{-n}_{x}|_{E^{+}}\right\| \leq C{\rm e}^{-\chi n},\quad \left\|D\varphi^{n}_{x}|_{E^{+}}\right\| \geq C^{-1}{\rm e}^{\chi n}$,
	\item  $\left\|D\varphi^{n}_{x}|_{E^{-}}\right\| \leq C{\rm e}^{-\chi n},\quad \left\|D\varphi^{-n}_{x}|_{E^{-}}\right\| \geq C^{-1}{\rm e}^{\chi n}$,
	\item $\measuredangle\left(E^{+}_{x}, E^{-}_{x}\right)\geq C^{-1}$.
\end{itemize}
We call the set $\cR_{C,\chi}$ above a \emph{weak Pesin set}.

Given $\varepsilon>0$, we define $\cR_{C,\chi,\varepsilon}\subset {\rm NUH}_{\chi}$ to be the set of $x$ such that for any $ n\in\NN, k\in\ZZ$,
\begin{itemize}
	\item $\left\|D\varphi^{-n}_{\varphi^{k}(x)}|_{E^{+}}\right\| \leq C{\rm e}^{|k|\varepsilon}{\rm e}^{-\chi n},\quad\left\|D\varphi^{n}_{\varphi^{k}(x)}|_{E^{+}}\right\| \geq C^{-1}{\rm e}^{-|k|\varepsilon}{\rm e}^{\chi n}$,
	\item $\left\|D\varphi^{n}_{\varphi^{k}(x)}|_{E^{-}}\right\| \leq C{\rm e}^{|k|\varepsilon}{\rm e}^{-\chi n},\,\left\|D\varphi^{-n}_{\varphi^{k}(x)}|_{E^{-}}\right\| \geq C^{-1}{\rm e}^{-|k|\varepsilon}{\rm e}^{\chi n}$,
	\item $\measuredangle\left(E^{+}_{\varphi^{k}(x)}, E^{-}_{\varphi^{k}(x)}\right)\geq C^{-1}{\rm e}^{-|k|\varepsilon}$.
\end{itemize}
The set $\cR_{C,\chi,\varepsilon}$ is commonly called a \emph{Pesin set}. Note that by definition, both $\cR_{C,\chi,\varepsilon}$ and $\cR_{C,\chi}$ are compact, and $\cR_{C,\chi,\varepsilon}\subset\cR_{C,\chi}$.

For any $x\in \mathcal{R}$, the \emph{global strong stable and unstable manifolds} at $x$ defined by $$W^{ss}(x):=\left\{y\in M:~\limsup_{t\to+\infty}\frac{1}{t}\log d\left(\varphi^{t}(x),\varphi^{t}(y)\right)<0\right\},$$ $$W^{uu}(x):=\left\{y\in M:~\limsup_{t\to+\infty}\frac{1}{t}\log d\left(\varphi^{-t}(x),\varphi^{-t}(y)\right)<0\right\}$$ are injectively immersed sub-manifolds with $T_{x}W^{ss}(x)=E^{-}_{x}, T_{x}W^{uu}(x)=E^{+}_{x}$. We define the corresponding \emph{local strong stable and unstable manifolds} by $$W^{*}_{r}(x):= \,\text{ connected part of}\,\,W^{*}(x)\cap B(x,r)\text{ containing }x,\quad *=ss,uu$$ where $B(x,r)$ is the ball centered at $x$ with radius $r$. Define the \emph{local stable and unstable manifolds} by $$W^{s}_{r}(x):= \bigcup_{t\in(-r,r)}\varphi^{t}\left(W^{ss}_{r}(x)\right),\quad W^{u}_{r}(x):= \bigcup_{t\in(-r,r)}\varphi^{t}\left(W^{uu}_{r}(x)\right).$$ By the classical Pesin theory, these local ones are embedded sub-manifolds for sufficiently small $r$ (which depends on $x$) whose geometries are uniformly controlled like in the uniformly hyperbolic case.

Given two sub-manifolds $D_{1}$ and $D_{2}$, let $D_{1}\pitchfork D_{2}$ denote the set of $x\in D_{1}\cap D_{2}$ such that $$T_{x}M=T_{x}D_{1}+T_{x}D_{2}.$$ Note that here because of the existence of the flow direction, unlike the case of diffeomorphisms, we do not assume the above sum is a direct sum.
\begin{Theorem}\label{Thm:Pesin}
	Given numbers $C,\chi,\varepsilon$, there is a number $r>0$ such that the local strong stable  and unstable manifolds $W^{ss}_{r}(x), W^{uu}_{r}(x)$ are well-defined for every $x\in \cR_{C,\chi,\varepsilon}$. And these manifolds are continuous in the $C^{1}$ topology in the sense that if $x_{n}\in  \cR_{C,\chi,\varepsilon}$ with $ x_{n}\to x$, then $$d_{C^{1}}(W^{*}_{r}(x_{n}), W^{*}_{r}(x))\to 0,\quad *=ss, uu, s, u.$$ Consequently, there is a number $\tau>0$ such that for any $x,y\in \cR_{C,\chi,\varepsilon}$, if $d(x,y)<\tau$, then $W_{r}^s(x)\pitchfork W_{r}^u(y)\neq\emptyset$.
\end{Theorem}
\begin{proof}
	These properties of  the local strong stable  and unstable manifolds are well known. For example, we refer to Chapter 7 and 8 in the book of Barreira-Pesin \cite{BPe07} for more detailed information. The existence of these manifolds is from \cite[Theorem 7.7.1]{BPe07}. The size estimation is from  \cite[Item 3 in Chapter 8.1]{BPe07}. The continuity is from  \cite[Item 4 in Chapter 8.1]{BPe07}. As a consequence of the continuity and the fact that the angle between $W^{uu}_{r}(x), W^{ss}_{r}(x)$ is uniformly bounded away from zero for all $x\in \cR_{C,\chi,\varepsilon}$, we can find some $\tau\ll r$ such that for any $x,y\in \cR_{C,\chi,\varepsilon}$ with $d(x,y)<\tau$, we have  $W_{r}^s(x)\pitchfork W_{r}^u(y)\neq\emptyset$.
\end{proof}
\begin{Remark}
For simplicity, we define the weak Pesin set and the Pesin set in discrete time ($n \in \mathbb{N}$) rather than continuous time ($t \in \mathbb{R}$). The strong stable and unstable manifolds for the time-one map coincide with those of the flow, allowing us to focus mainly on the discrete time setting.
\end{Remark}

\subsection{Comparison between weak Pesin sets and Pesin sets}
The following result demonstrates that if a weak Pesin set has large measure with respect to an invariant measure, then the corresponding Pesin set also has large measure. More importantly, this property holds uniformly, meaning that the parameters in the definition of Pesin sets are independent of the invariant measures.
\begin{Theorem}\label{Pro:compare-two-sets}
	Let $X$ be a $C^r(r>1)$ non-singular vector field over a three-dimensional compact manifold $M$ without boundary. For any small $\varepsilon,\alpha>0$, there is $\delta>0$ such that for any $C>1, \chi>0$ and any $\varphi^{1}$-invariant measure $\nu$, if $$\nu\left(\cR_{C,\chi}\right)>1-\delta,$$ then $$\nu\left(\cR_{C,\chi, \varepsilon}\right)>1-\alpha.$$
\end{Theorem}
\begin{Remark}
The above theorem remains valid for $C^{r}$ non-singular flows in arbitrary dimensions. For simplicity, we state it only in the three-dimensional setting, which is sufficient for the purposes of this paper.
\end{Remark}

We begin by presenting some general results on the comparison of sets that resemble $\cR_{C,\chi}$ and $\cR_{C,\chi,\varepsilon}$.
\begin{Lemma}\label{SPR for function}
	Let $f$ be a homeomorphism on a metric space $X$ and let $A: X\to (0, +\infty)$ be a measurable function such that 
	\begin{itemize}
		\item for any $x\in X$, $$\lim_{n\to+\infty}\frac{1}{n}\log A(f^{n}(x))=0,$$
		\item there is a constant $K>1$ such that for any $x\in X$, $$K^{-1}\leq\frac{A(f(x))}{A(x)}\leq K.$$
	\end{itemize}
	Then for any $0<\varepsilon<1$, any $t>0$ and any invariant measure $\nu$, $$\nu\left(\{x:~\overline{A}_{\varepsilon}(x)> t\}\right)\leq \frac{K+1}{\varepsilon}\cdot\nu\left(\{x:~A(x)> t\}\right),\quad \nu\left(\{x:~\underline{A}_{\varepsilon}(x)< t\}\right)\leq \frac{K+1}{\varepsilon}\cdot\nu\left(\{x:~A(x)<t\}\right)$$ where $$\overline{A}_{\varepsilon}(x):=\sup_{n\geq 0}{\rm e}^{-n\varepsilon}A(f^{n}(x)),\quad \underline{A}_{\varepsilon}(x):=\inf_{n\geq 0}{\rm e}^{n\varepsilon}A(f^{n}(x)).$$
\end{Lemma}
\begin{proof} We first prove the case $\nu(\{x:~\overline{A}_{\varepsilon}(x)> t\})\leq \frac{K+1}{\varepsilon}\cdot\nu(\{x:~A(x)> t\})$. 
	$$\begin{aligned}
		\nu\left(\{x:~\overline{A}_{\varepsilon}(x)> t\}\right)
		&\leq \sum_{n=0}^{+\infty}\nu\left(\{x:~A(f^{n}(x))> t\cdot {\rm e}^{n\varepsilon},\,A(f^{n+1}(x))\leq t\cdot {\rm e}^{(n+1)\varepsilon}\}\right)\\
		&=\sum_{n=0}^{+\infty}\nu\left(\{x:~A(x)> t\cdot {\rm e}^{n\varepsilon},\,A(f(x))\leq t\cdot {\rm e}^{(n+1)\varepsilon}\}\right)\\
		&\leq\sum_{n=0}^{+\infty}\nu\left(\{x:~A(x)> t\cdot {\rm e}^{n\varepsilon},\,A(x)\leq K\cdot t\cdot {\rm e}^{(n+1)\varepsilon}\}\right)\\
		&\leq\sum_{n=0}^{+\infty}\nu\left(\{x:~A(x)> t\cdot {\rm e}^{n\varepsilon},\,A(x)\leq t\cdot {\rm e}^{(n+\frac{K}{\varepsilon})\varepsilon}\}\right)\\
		&\leq \frac{K+1}{\varepsilon}\cdot\nu\left(\{x:~A(x)> t\}\right).	
	\end{aligned}$$	Note that the fourth inequality holds because $\ln K + \varepsilon < K$ for any $K>1$ and $0<\varepsilon<1$, which implies $K\cdot {\rm e}^{(n+1)\varepsilon} \leq {\rm e}^{(n+\frac{K}{\varepsilon})\varepsilon}$. For the last step, we used the observation that  the preceding level sets $\{x:~A(x)> t\cdot {\rm e}^{n\varepsilon},\,A(x)\leq t\cdot {\rm e}^{(n+\frac{K}{\varepsilon})\varepsilon}\}$ can cover $\{x:~A(x)> t\}$ at most $\frac{K}{\varepsilon} + 1$ ($\leq \frac{K+1}{\varepsilon}$) times.

		The proof for the other case is analogous. But for completeness, we also provide detailed arguments.
		$$\begin{aligned}
		\nu\left(\{x:~\underline{A}_{\varepsilon}(x)<t\}\right)
		&\leq \sum_{n=0}^{+\infty}\nu\left(\{x:~A(f^{n}(x))<t\cdot {\rm e}^{-n\varepsilon},\,A(f^{n+1}(x))\geq t\cdot {\rm e}^{-(n+1)\varepsilon}\}\right)\\
		&=\sum_{n=0}^{+\infty}\nu\left(\{x:~A(x)<t\cdot {\rm e}^{-n\varepsilon},\,A(f(x))\geq t\cdot {\rm e}^{-(n+1)\varepsilon}\}\right)\\
		&\leq\sum_{n=0}^{+\infty}\nu\left(\{x:~A(x)<t\cdot {\rm e}^{-n\varepsilon},\,A(x)\geq K^{-1}\cdot t\cdot {\rm e}^{-(n+1)\varepsilon}\}\right)\\
		&\leq\sum_{n=0}^{+\infty}\nu\left(\{x:~A(x)<t\cdot {\rm e}^{-n\varepsilon},\,A(x)\geq t\cdot {\rm e}^{-(n+\frac{K}{\varepsilon})\varepsilon}\}\right)\\
		&\leq \frac{K+1}{\varepsilon}\cdot\nu\left(\{x:~A(x)< t\}\right).	
	\end{aligned}$$	
\end{proof}

Now we proceed to the proof of Theorem \ref{Pro:compare-two-sets}. The key idea is to apply Lemma \ref{SPR for function} within an appropriately chosen framework.
\begin{proof}[Proof of Theorem \ref{Pro:compare-two-sets}]
	We note that there are several conditions in the definition of the sets $\cR_{C,\chi}$ and $\cR_{C,\chi,\varepsilon}$. For simplicity, we focus on just two of these conditions to illustrate the main idea of the proof and the theorem follows by adding more conditions, which can be handled in the same manner. We define 
	$$\Lambda_{C,\chi}=\left\{x\in {\rm NUH}_{\chi}:~\forall n\in\NN, \left\|D\varphi^{-n}_{x}|_{E^{+}}\right\| \leq C{\rm e}^{-\chi n}\right\},$$ $$\Lambda_{C,\chi,\varepsilon}=\left\{x\in {\rm NUH}_{\chi}:~\forall n\in\NN, k\in\NN, \left\|D\varphi^{-n}_{\varphi^{k}(x)}|_{E^{+}}\right\| \leq C{\rm e}^{k\varepsilon}{\rm e}^{-\chi n}\right\},$$	
	$$\Delta_{C,\chi}=\left\{x\in {\rm NUH}_{\chi}:~\measuredangle\left(E^{+}_{x}, E^{-}_{x}\right)\geq C^{-1}\right\},$$ $$\Delta_{C,\chi,\varepsilon}=\left\{x\in {\rm NUH}_{\chi}:~\forall k\in\NN,\,\measuredangle\left(E^{+}_{\varphi^{k}(x)}, E^{-}_{\varphi^{k}(x)}\right)\geq C^{-1}{\rm e}^{-k\varepsilon}\right\}.$$
	We first prove the following result w.r.t. $\varphi^{1}$-ergodic measures and then extend it to $\varphi^{1}$-invariant measures (which proves the theorem). 
	\begin{Claim}For any small $\varepsilon,\alpha>0$, there is $\delta>0$ such that for any $C>1, \chi>0$ and any $\varphi^{1}$-ergodic measure $\nu$,
		\begin{enumerate}
			\item if $\nu\left(\Lambda_{C,\chi}\right)> 1-\delta$, then $\nu\left(\Lambda_{C,\chi,\varepsilon}\right)> 1-\alpha$,
			\item if $\nu\left(\Delta_{C,\chi}\right)> 1-\delta$, then $\nu\left(\Delta_{C,\chi,\varepsilon}\right)> 1-\alpha$.
		\end{enumerate}
		
	\end{Claim}
	\begin{proof}
		We note that ${\rm NUH}_{\chi}\subset\cR$ is an invariant subset. Define a function $A: {\rm NUH}_{\chi}\to (0, +\infty)$ by $$A(x)=\sup_{n\geq 0} \left\|D\varphi^{-n}_{x}|_{E^{+}}\right\|\cdot {\rm e}^{\chi n}.$$ We note by definition that $A$ is a well-defined measurable function. Moreover we have the following properties of $A$:
		\begin{itemize}
			\item there is a constant $K>1$ depending only on the flow such that $$K^{-1}\leq\frac{A(\varphi^{1}(x))}{A(x)}\leq K,\quad\forall\, x\in {\rm NUH}_{\chi},$$ 
			\item for any $C>1$, $\Lambda_{C,\chi}=\left\{x\in{\rm NUH}_{\chi}:~ A(x)\leq C\right\}$,
			\item by Lemma \ref{Mane Temper function} in Appendix, there is an invariant subset $\Omega$ of ${\rm NUH}_{\chi}$ with total measure (i.e., $\mu(\Omega)=1$ for any ergodic measure $\mu$ with $\mu({\rm NUH}_{\chi})=1$) such that for any $x\in\Omega$ ,\, $$\lim_{n\to+\infty}\frac{1}{n}\log A(\varphi^{n}(x))=0.$$ For convenience, by replacing ${\rm NUH}_{\chi}$ with the total measure set $\Omega$, we may assume the convergence above holds for any $x\in{\rm NUH}_{\chi}$. This assumption is justified because we are only concerned with ergodic measures $\nu$ satisfying $\nu(\Lambda_{C,\chi})>0$, in which case ergodicity ensures that $\nu(\Omega) = \nu({\rm NUH}_{\chi}) = 1$. This is the main reason why we begin by considering the ergodic case, as stated in this claim.
			
		\end{itemize}
		By the last property above, the measurable function $A_{\varepsilon}: {\rm NUH}_{\chi}\to (0, +\infty)$ $$A_{\varepsilon}(x):=\sup_{k\geq 0}{\rm e}^{-k\varepsilon}A(\varphi^{k}(x))$$ is well-defined. We note the following properties:
		\begin{itemize}
			\item $A_{\varepsilon}(x)\geq A(x)$, for any $x\in{\rm NUH}_{\chi}$,
			\item for any $C>1$ and any small $\varepsilon>0$, $\Lambda_{C,\chi,\varepsilon}=\left\{x\in{\rm NUH}_{\chi}:~ A_{\varepsilon}(x)\leq C\right\}$.
		\end{itemize}
		By Lemma \ref{SPR for function}, for any $t>0$ and any $\varphi^{1}$-ergodic measure $\nu$, 
		\begin{equation}\label{consequence of appendix}
			\nu\left(\{x\in{\rm NUH}_{\chi}:~A_{\varepsilon}(x)> t\}\right)\leq \frac{K+1}{\varepsilon}\cdot\nu\left(\{x\in{\rm NUH}_{\chi} :~A(x)> t\}\right).
		\end{equation}
		
		Given small $\alpha>0$, let $\delta:=\frac{\alpha\varepsilon}{K+1}$. For any $\varphi^{1}$-ergodic measure $\nu$ with $\nu(\Lambda_{C,\chi})>1-\delta$, we have 
		\begin{equation}\label{property of A(x)}
			\nu\left(\{x\in{\rm NUH}_{\chi}:~ A(x)> C\}\right)=\nu\left(M\setminus \Lambda_{C,\chi}\right)\leq \delta.
		\end{equation}
		
		Recall that 
		\begin{equation}\label{property of A varepsilon(x)}
			\nu\left(\{x\in{\rm NUH}_{\chi}:~ A_\varepsilon(x)> C\}\right)=\nu\left(M\setminus \Lambda_{C,\chi,\varepsilon}\right).
		\end{equation}
		Combining (\ref{consequence of appendix}),(\ref{property of A(x)}) and (\ref{property of A varepsilon(x)}), we have $$\nu\left(M\setminus \Lambda_{C,\chi,\varepsilon}\right)\leq \alpha,$$ and consequently, $$\nu\left( \Lambda_{C,\chi,\varepsilon}\right)> 1-\alpha.$$
		
		Next we turn to the condition of angles. The approach is analogous to the case of $\Lambda_{C,\varepsilon}$ discussed above, with Lemma \ref{SPR for function} playing a central role. Denote the angle functions $B, B_{\varepsilon}: {\rm NUH}_{\chi}\to (0, +\infty)$ by $$B(x):=\measuredangle\left(E^{+}_{x}, E^{-}_{x}\right),\quad B_{\varepsilon}(x):=\inf_{k\geq 0}\measuredangle\left(E^{+}_{\varphi^{k}(x)}, E^{-}_{\varphi^{k}(x)}\right)\cdot{\rm e}^{k\varepsilon}.$$We note the following properties of $B, B_{\varepsilon}$:
		\begin{itemize}
			\item there is a constant $K'>1$ depending only on the flow such that $$\left(K'\right)^{-1}\leq\frac{B(\varphi^{1}(x))}{B(x)}\leq K',\quad\forall\, x\in {\rm NUH}_{\chi},$$ 
			\item for any $C>1$, $\Delta_{C,\chi}=\{x\in{\rm NUH}_{\chi}:~ B(x)\geq C^{-1}\}$,
			\item for any $C>1$ and any small $\varepsilon>0$, $\Delta_{C,\chi,\varepsilon}=\left\{x\in{\rm NUH}_{\chi}:~ B_{\varepsilon}(x)\geq C^{-1}\right\}$,
			\item by Oseledets Theorem, for any $x\in {\rm NUH}_{\chi}$,\, $\lim_{n\to+\infty}\frac{1}{n}\log B(\varphi^{n}(x))=0$.
		\end{itemize}
		Again we apply Lemma \ref{SPR for function} to the function $B$ and we get that for any $t>0$ and any $\varphi^{1}$-ergodic measure $\nu$, 	$$\nu\left(\{x\in{\rm NUH}_{\chi}:~B_{\varepsilon}(x)<C^{-1}\}\right)\leq \frac{K'+1}{\varepsilon}\cdot\nu\left(\{x\in{\rm NUH}_{\chi} :~B(x)<C^{-1}\}\right).$$ Given $\alpha>0$ small, let $\delta:=\frac{\alpha\varepsilon}{K'+1}$. For any $\varphi^{1}$-ergodic measure $\nu$ with $\nu(\Delta_{C,\chi})>1-\delta$, we have the following properties like in the previous case: $$\nu\left(\{x\in{\rm NUH}_{\chi}:~ B(x)< C^{-1}\}\right)=\nu\left(M\setminus\Delta_{C,\chi}\right)\leq \delta,$$ $$\nu\left(\{x\in{\rm NUH}_{\chi}:~ B_\varepsilon(x)< C^{-1}\}\right)=\nu\left(M\setminus\Delta_{C,\chi,\varepsilon}\right).$$ This gives $$\nu\left(M\setminus\Delta_{C,\chi,\varepsilon}\right)\leq \alpha,$$ and consequently, $$\nu\left(\Delta_{C,\chi,\varepsilon}\right)> 1-\alpha.$$
		
		The proof of the claim is complete.
	\end{proof}
	Next, we will use the ergodic decomposition to show that the claim above holds for any $\varphi^{1}$-invariant measure which proves the theorem. We will only show the statement for $\Lambda_{C,\chi}$ as the statement for $\Delta_{C,\chi}$ can be proved analogously. 
	
	Given any small $\varepsilon,\alpha>0$, the claim above provides a small number $\widetilde{\delta}>0$ such that for any $C>1, \chi>0$ and any $\varphi^{1}$-ergodic measure $m$, if $$m\left(\Lambda_{C,\chi}\right)> 1-\widetilde{\delta},$$ then $$m\left(\Lambda_{C,\chi,\varepsilon}\right)> 1-\frac{\alpha}{2}.$$ We may further assume $\widetilde{\delta}>0$ is small enough such that $$(1-\widetilde{\delta})\cdot(1-\frac{\alpha}{2})>(1-\alpha).$$ 
	
	Let $$\delta:=\widetilde{\delta}^{2}.$$ Given any $\varphi^{1}$-invariant $\nu$ with $\nu(\Lambda_{C,\chi})> 1-\delta=1-\widetilde{\delta}^{2}$, we consider the ergodic decomposition of $\nu$ and let $\widetilde{\nu}$ be the corresponding probability measure associated with $\nu$ on the space of ergodic measures (w.r.t. $\varphi^{1}$). We note that there is a collection $\mathcal{V}$ of the ergodic components of $\nu$ with $\widetilde{\nu}(\mathcal{V})>1-\widetilde{\delta}$ such that for each component $m\in \mathcal{V}$, $$m\left(\Lambda_{C,\chi}\right)>1-\widetilde{\delta}.$$ Because otherwise, we would have $$\nu\left(\Lambda_{C,\chi}\right)=\int m\left(\Lambda_{C,\chi}\right)\,d\,\widetilde{\nu}(m)\leq 1-\widetilde{\delta}\cdot\widetilde{\delta}=1-\delta$$ which leads to a contradiction. Hence, we have $$\nu\left(\Lambda_{C,\chi,\varepsilon}\right)=\int m\left(\Lambda_{C,\chi,\varepsilon}\right)\,d\,\widetilde{\nu}(m)\geq\int_{\mathcal{V}} m\left(\Lambda_{C,\chi,\varepsilon}\right)\,d\,\widetilde{\nu}(m)> (1-\widetilde{\delta})\cdot(1-\frac{\alpha}{2})>1-\alpha.$$

	 The proof of Theorem \ref{Pro:compare-two-sets} is complete.
\end{proof}

\section{Finiteness of MMEs}
\subsection{Uniform largeness of the weak Pesin sets}
Recall that the weak Pesin set $\cR_{C,\chi}$ in Section \ref{size of local manifolds} is the set of $x\in{\rm NUH}_{\chi}$ such that for any $ n\in\NN$, 

\begin{itemize}
	\item  $\left\|D\varphi^{-n}_{x}|_{E^{+}}\right\| \leq C{\rm e}^{-\chi n},\quad \left\|D\varphi^{n}_{x}|_{E^{+}}\right\| \geq C^{-1}{\rm e}^{\chi n}$,
	\item  $\left\|D\varphi^{n}_{x}|_{E^{-}}\right\| \leq C{\rm e}^{-\chi n},\quad \left\|D\varphi^{-n}_{x}|_{E^{-}}\right\| \geq C^{-1}{\rm e}^{\chi n}$,
	\item $\measuredangle\left(E^{+}_{x}, E^{-}_{x}\right)\geq C^{-1}$.
\end{itemize}
We next establish that weak Pesin sets have uniformly large measure for invariant measures with large Lyapunov exponents, including in particular the MMEs. This result fundamentally follows from the continuity of Lyapunov exponents (see Theorem \ref{continuity of exponents for flow}). 

\begin{Theorem}\label{General result of Thm:limit measure carries large exponent everywhere}
Let $X$ be a $C^r(r>1)$ non-singular vector field over a three-dimensional compact manifold $M$ without boundary. Assume that $\mu_n$ is a sequence of $\varphi^{1}$-invariant hyperbolic measures with $\mu_{n}\to\mu$ for some $\varphi^{1}$-invariant hyperbolic measure $\mu$ such that
	\begin{itemize}
		\item $\mu({\rm NUH}_{2\chi})=1$ for some $\chi>0$,		
		\item $\hat{\mu}^{+}_{n}, \hat{\mu}^{-}_{n}$ converge respectively to $\hat{\mu}^{+}, \hat{\mu}^{-}$.
	\end{itemize}
	Then for any $\delta>0$, there are $C>0$ and $N\in\NN$ such that for any $n>N$,
	$$\mu_n\left(\cR_{C,\chi}\right)>1-\delta.$$
\end{Theorem}
In particular, we obtain the following result for MMEs, which is our primary focus.
\begin{Theorem}\label{Thm:limit measure carries large exponent everywhere}
		Let $X$ be a $C^r(r>1)$ non-singular vector field with $h_{\rm top}(X)>\frac{1}{r}\min\left\{\lambda^{+}(X), \lambda^{-}(X)\right\}$ over a three-dimensional compact manifold $M$ without boundary. Assume that $\mu_n$ is a sequence of $\varphi^{1}$-ergodic MMEs with $\mu_{n}\to\mu$ for some $\varphi^{1}$-invariant measure $\mu$. Write $\chi:=\frac{h_{\rm top}(X)}{3}$. Then for any $\delta>0$, there are $C>0$ and $N\in\NN$ such that for any $n>N$,
	$$\mu_n\left(\cR_{C,\chi}\right)>1-\delta.$$
\end{Theorem}
\begin{proof}[Proof of Theorem \ref{Thm:limit measure carries large exponent everywhere} assuming Theorem \ref{General result of Thm:limit measure carries large exponent everywhere}]
	By Theorem \ref{continuity of exponents for flow}, $\mu$ is a MME. By the ergodic decomposition and the harmonic property of entropy, every ergodic component of $\mu$ is also a MME. Hence, by Ruelle's inequality, $\mu$-almost every point has one positive exponent which is strictly larger than $2\chi$, one zero exponent which is on the flow direction and one negative exponent which is strictly less than $-2\chi$. This implies $\mu({\rm NUH}_{2\chi})=1$ . Moreover, by Theorem \ref{continuity of exponents for flow}, $\hat{\mu}^{+}_{n}, \hat{\mu}^{-}_{n}$ converge respectively to $\hat{\mu}^{+}, \hat{\mu}^{-}$. By Theorem \ref{General result of Thm:limit measure carries large exponent everywhere} above, we get the conclusion.
\end{proof}

Next we prove Theorem \ref{General result of Thm:limit measure carries large exponent everywhere}. The following result can be viewed as an abstract version of Theorem \ref{General result of Thm:limit measure carries large exponent everywhere} above for the lifted dynamics.

\begin{Lemma}\label{Lem: Pliss-Like Lemma1}
	Let $f$ be a homeomorphism on a compact metric space $X$ and let $\omega: X\to \mathbb{R}$ be a continuous function. Let $\mu$ be an invariant measure such that for $\mu$-almost every $x$,  $$\lim_{n\to+\infty}\frac{1}{n}\sum_{i=0}^{n-1}\omega\left(f^{i}(x)\right)<-2\chi<0.$$	
	Then for any $\delta>0$, there are $C>0$ and a neighborhood $\mathcal{U}$ of $\mu$ such that for any invariant measure $\nu\in \mathcal{U}$, we have 
	$$\nu(\Omega_{C})>1-\delta$$ where $$\Omega_{C}:=\left\{x:~\forall n\in\NN, \sum_{i=0}^{n-1}\omega\left(f^{i}(x)\right)\leq -\chi \cdot n+C\right\}.$$
\end{Lemma}
\begin{proof}
	It is sufficient to prove the result for all sufficiently small $\delta$.
	By Birkhoff's ergodic theorem, given any small number $\delta$ with $$\delta< \frac{2\chi}{3}\big/\left(\max_{x}|\omega(x)|+1\right) \quad\text{and}\quad -\frac{3}{2}\cdot (1-\frac{\delta}{2})<-\frac{4}{3},$$ there is some $N$ large enough such that $\mu(\Lambda_{N})> 1-\frac{\delta^{4}}{4}$ where  the open (by the continuity of $\omega$) subset $\Lambda_{N}$ is defined by $$\Lambda_{N}:=\left\{x:~ \sum_{i=0}^{N-1}\omega\left(f^{i}(x)\right)<-\frac{3}{2}\chi \cdot N\right\}.$$
	Let $\mathcal{U}$ be a neighborhood of $\mu$ such that for any $\nu\in \mathcal{U}$, $\nu(\Lambda_{N})>1-\frac{\delta^{4}}{4}$ (since $\Lambda_{N}$ is open). Given any $\nu\in \mathcal{U}$, we consider the ergodic decomposition of $\nu$ w.r.t. $f^{N}$ and let $\widetilde{\nu}$ be the corresponding probability measure associated with $\nu$ on the space of $f^{N}$-ergodic measures. We note that there is a collection $\mathcal{V}$ of the ergodic components of $\nu$ with $\widetilde{\nu}(\mathcal{V})>1-\delta^{2}$ such that for each component $m\in \mathcal{V}$, $$m(\Lambda_{N})>1-\frac{\delta^{2}}{4}.$$ Because otherwise, we would have $$\nu(\Lambda_{N})=\int m(\Lambda_{N})\,d\,\widetilde{\nu}(m)\leq 1-\delta^{2}\cdot\frac{\delta^{2}}{4}=1-\frac{\delta^{4}}{4}$$ which is a contradiction. 
	
	Given any $f^{N}$-ergodic component $m\in\mathcal{V}$, since $m(\Lambda_{N})>1-\frac{\delta^{2}}{4}$, by Lemma \ref{Lem:abstract-ergodic} (w.r.t. $f^{N}$) in Appendix, we have $$m(\widetilde{\Lambda}_{N})\ge 1-\frac{\delta}{2}$$ where $$\widetilde{\Lambda}_{N}:=\left\{x: ~\forall \ell\ge1,~\frac{1}{\ell}\#\left\{k:~0\le k<\ell,~f^{kN}(x)\in \Lambda_{N}\right\}\ge 1-\frac{\delta}{2}\right\}.$$ Given $n\in\mathbb{N}$, let $l$ be the integer such that $n-1=l\cdot N+r$ where $0\leq r<N$. Given any $x\in \widetilde{\Lambda}_{N}$, let $\mathbb{L}$ denote the collection of integers $k$ in $[0,l-1]$ such that $f^{kN}(x)\in\Lambda_{N}$. Write $\mathbb{L}^{c}:=[0,l-1]\setminus \mathbb{L}$. By the result above, we have $\# \mathbb{L}\geq l(1-\frac{\delta}{2})$ and $\# \mathbb{L}^{c}\leq l\cdot\frac{\delta}{2}$. Let $C=10N\cdot(\max_{x}|\omega(x)|+1)$. We have 
	$$\begin{aligned}
		\sum_{i=0}^{n-1}\omega\left(f^{i}(x)\right)
		&\leq\sum_{i=0}^{lN-1}\omega\left(f^{i}(x)\right)+\frac{C}{2}\\
		&=\sum_{k\in \mathbb{L}}\sum_{i=0}^{N-1}\omega\left(f^{kN+i}(x)\right)+\sum_{k\in\mathbb{L}^{c}}\sum_{i=0}^{N-1}\omega\left(f^{kN+i}(x)\right)+\frac{C}{2}\\
		&\leq-\frac{3}{2}\chi N\cdot l(1-\frac{\delta}{2})+N\cdot\left(\max_{x}|\omega(x)|+1\right)\cdot l\cdot\frac{\delta}{2}+\frac{C}{2}\\
		&\leq-\frac{4}{3}\chi\cdot lN+\frac{1}{3}\chi\cdot lN+\frac{C}{2}\\
		&= -\chi \cdot lN+\frac{C}{2}\\
		&\leq -\chi \cdot n+C.
	\end{aligned}$$	
	Hence $\widetilde{\Lambda}_{N}\subset \Omega_{C}$. Then, we have $$m(\Omega_{C})\geq m(\widetilde{\Lambda}_{N})\geq 1-\frac{\delta}{2}.$$ Consequently, $$\nu(\Omega_{C})=\int m(\Omega_{C})\,d\,\widetilde{\nu}(m)\geq (1-\delta^{2})\cdot(1-\frac{\delta}{2})= 1+\frac{\delta^{3}}{2}-\delta^{2}-\frac{\delta}{2}>1-\delta.$$

\end{proof}

We now use Lemma \ref{Lem: Pliss-Like Lemma1} to establish Theorem \ref{General result of Thm:limit measure carries large exponent everywhere}.
\begin{proof}[Proof of Theorem \ref{General result of Thm:limit measure carries large exponent everywhere}]
	Given $C>1$, we define $\Omega_{C,\chi}\subset {\rm NUH}_{\chi}$ to be the set of $x$ such that for any $ n\in\NN$, 	
	\begin{itemize}
		\item  $\left\|D\varphi^{-n}_{x}|_{E^{+}}\right\| \leq C{\rm e}^{-\chi n},\quad \left\|D\varphi^{n}_{x}|_{E^{+}}\right\| \geq C^{-1}{\rm e}^{\chi n}$,
		\item  $\left\|D\varphi^{n}_{x}|_{E^{-}}\right\| \leq C{\rm e}^{-\chi n},\quad \left\|D\varphi^{-n}_{x}|_{E^{-}}\right\| \geq C^{-1}{\rm e}^{\chi n}$.
	\end{itemize}
Here, note that $\Omega_{C,\chi} $ is defined by omitting the angle condition in $\cR_{C,\chi}$. We aim to show that the measure of $\Omega_{C,\chi}$ can be made arbitrarily close to one. Specifically, we state the following claim:
\begin{Claim}
	For any small $\delta>0$, there are $\widetilde{C}>1$ and $N\in\NN$ such that for any $n>N$, $$\mu_{n}\left(\Omega_{\widetilde{C},\chi}\right)> 1-\delta.$$	
\end{Claim}
	
\begin{proof}	
		Note that the definition of $\Omega_{C,\chi}$ involves four conditions. It suffices to show that each condition can be satisfied by sets whose measure is arbitrarily close to one. For simplicity, we focus on the conditions related to $E^{+}$, as those for $E^{-}$ can be handled in a similar manner. To be precise, we define	$$\Lambda_{\widetilde{C},\chi}:=\left\{x\in {\rm NUH}_{\chi}:~\forall n\in\NN, \|D\varphi^{-n}_{x}|_{E^{+}}\| \leq \widetilde{C}{\rm e}^{-\chi n}\right\},$$ $$\Delta_{\widetilde{C},\chi}:=\left\{x\in {\rm NUH}_{\chi}:~\forall n\in\NN, \|D\varphi^{n}_{x}|_{E^{+}}\| \geq \widetilde{C}^{-1}{\rm e}^{\chi n}\right\}.$$ We will prove that for any small $\delta>0$, there are $\widetilde{C}>0$ and $N\in\NN$ such that for any $n>N$, $$\mu_{n}\left(\Lambda_{\widetilde{C},\chi}\right)> 1-\delta, \quad\mu_{n}\left(\Delta_{\widetilde{C},\chi}\right)> 1-\delta.$$

		Let $\hat{\Lambda}_{\widetilde{C},\chi}, \hat{\Delta}_{\widetilde{C},\chi}$ denote the corresponding unstable lifted sets in $\hat{M}$, i.e., $$\hat{\Lambda}_{\widetilde{C},\chi}:=\left\{(x,E^{+}_{x}):~x\in \Lambda_{\widetilde{C},\chi}\right\},\quad \hat{\Delta}_{\widetilde{C},\chi}:=\left\{(x,E^{+}_{x}):~x\in \Delta_{\widetilde{C},\chi}\right\}.$$ Note that  it suffices to prove that $$\hat{\mu}_n^+\left(\hat\Lambda_{\widetilde{C},\chi}\right)>1-\delta,\quad \hat{\mu}_n^+\left(\hat\Delta_{\widetilde{C},\chi}\right)>1-\delta$$ for some $\widetilde{C}>0$. Recall that $\hat{\mu}^{+}$ is the unstable lift of $\mu$ which has full measure on the points $\left\{(x,E^{+}_{x})\right\}$. Define the continuous function on $\hat{M}$ by $$\omega_{1}(x,E):=\log \left\|D\varphi^{-1}_{x}|_{E}\right\|,\quad \omega_{2}(x,E):=-\log \left\|D\varphi^{1}_{x}|_{E}\right\|.$$ Recall that $\hat{\varphi}$ is the lifted flow of $\varphi$ on $\hat{M}$ defined by $$\hat{\varphi}^{t}(x, E):=\left(\varphi^{t}(x), D\varphi^{t}_{x}(E)\right).$$ For $\hat{\mu}^{+}$-almost every point $(x,E^{+}_{x})$, we have  $$\lim_{n\to+\infty}\frac{1}{n}\sum_{i=0}^{n-1}\omega_{1}\left(\hat{\varphi}^{-i}(x,E^{+}_{x})\right)=\lim_{n\to+\infty}\frac{1}{n}\log\left\|D\varphi^{-n}_{x}|_{E^{+}}\right\|<-2\chi<0,$$	$$\lim_{n\to+\infty}\frac{1}{n}\sum_{i=0}^{n-1}\omega_{2}\left(\hat{\varphi}^{i}(x,E^{+}_{x})\right)=\lim_{n\to+\infty}-\frac{1}{n}\log\left\|D\varphi^{n}_{x}|_{E^{+}}\right\|<-2\chi<0$$	
		and $$\hat{\Lambda}_{\widetilde{C},\chi}=\left\{\left(x,E^{+}_{x}\right):~\forall n\in\NN, \sum_{i=0}^{n-1}\omega_{1}\left(\hat{\varphi}^{-i}(x,E^{+}_{x})\right)\leq -\chi \cdot n+\log\widetilde{C}\right\},$$ $$\hat{\Delta}_{\widetilde{C},\chi}=\left\{\left(x,E^{+}_{x}\right):\,~\forall n\in\NN, \sum_{i=0}^{n-1}\omega_{2}\left(\hat{\varphi}^{i}(x,E^{+}_{x})\right)\leq -\chi \cdot n+\log\widetilde{C}\right\}.$$

		Applying Lemma \ref{Lem: Pliss-Like Lemma1} to $(\hat{M},\hat{\varphi}^{-1},\omega_{1}, \hat{\mu}^{+})$ and $(\hat{M},\hat{\varphi}^{1},\omega_{2}, \hat{\mu}^{+})$ respectively, there are a constant $\widetilde{C}>1$ and an integer $N\in\mathbb{N}$ such that for any $n>N$, $$\hat{\mu}_n^+\left(\hat{\Lambda}_{\widetilde{C},\chi}\right)>1-\delta,\quad \hat{\mu}_n^+\left(\hat{\Delta}_{\widetilde{C},\chi}\right)>1-\delta.$$
		Hence, 	$$\mu_{n}\left(\Lambda_{\widetilde{C},\chi}\right)=\hat{\mu}_n^+\left(\hat{\Lambda}_{\widetilde{C},\chi}\right)>1-\delta,\quad \mu_{n}\left(\Delta_{\widetilde{C},\chi}\right)=\hat{\mu}_n^+\left(\hat{\Delta}_{\widetilde{C},\chi}\right)>1-\delta.$$	
		
		The proof of the claim is complete.
\end{proof}		
Next we show $\mu_{n}(\cR_{C,\chi})> 1-\delta$ for some large $C>1$. To see this, by the continuity of the derivative $D\varphi^{1}$, there must be a constant $C>\widetilde{C}$ independent of $x$ such that for any $x\in \Omega_{\widetilde{C},\chi}$, $$\measuredangle\left(E^{+}_{x}, E^{-}_{x}\right)\geq C^{-1}.$$ Otherwise, at least two of the four conditions in the definition of $\Omega_{\widetilde{C},\chi}$ cannot be satisfied simultaneously (see Lemma~\ref{angle uniformly bounded from below} in Appendix). Hence, we have $$\Omega_{\widetilde{C},\chi}\subset \cR_{C,\chi}.$$ Therefore, for all $n>N$, $$\mu_{n}\left(\cR_{C,\chi}\right)\geq\mu_{n}\left(\Omega_{\widetilde{C},\chi}\right)> 1-\delta.$$

\end{proof}
\subsection{Proof of Theorem~\ref{Thm:finiteness C r}}\label{Proof of Theorem B}
Recall that for two sub-manifolds $D_{1}$ and $D_{2}$, $D_{1}\pitchfork D_{2}$ denotes the set of $x\in D_{1}\cap D_{2}$ such that $$T_{x}M=T_{x}D_{1}+T_{x}D_{2}.$$ Two ergodic hyperbolic measures $\mu,\nu$ are called \emph{homoclinically related} if there are two sets $A, B$ with $\mu(A), \nu(B)>0$ such that for any $x\in A, y\in B$, $$W^s(x)\pitchfork W^u(y)\neq\emptyset,\quad W^u(x)\pitchfork W^s(y)\neq\emptyset.$$ By ergodicity, it is equivalent to say that the above property holds for $\mu$-a.e. $x$ and $\nu$-a.e. $y$.

To prove the finiteness of MMEs, we use the following result of Buzzi-Crovisier-Lima \cite[Corollary 1.2]{BCL25}: 

There is \emph{at most one} ergodic hyperbolic measure $\mu$ which is homoclinically related to $\nu$ such that 
\begin{equation}\label{Buzzi-Crovisier-Lima Corollary}
	h(X,\mu)=\max\left\{h(X,\nu'):~  \nu' \text{ is homoclinically related to } \nu\right\}.
\end{equation}

\begin{proof}[Proof of Theorem~\ref{Thm:finiteness C r}]
We prove the theorem by contradiction. Assume that $X$ has infinitely many ergodic distinct MMEs: $\widetilde{\mu}_1, \widetilde{\mu}_2,\cdots, \widetilde{\mu}_n,\cdots$. Usually, each $\widetilde{\mu}_{n}$ is not ergodic but merely invariant under the time-one map $\varphi^{1}$. However, by Lemma \ref{ergodic decomposition of flow} in Appendix, we can find a sequence of $\varphi^{1}$-ergodic MMEs: $\mu_1,\mu_2,\cdots,\mu_n,\cdots$ such that for each $n$, 
\begin{equation}\label{ergodic component for time one map}
	\widetilde{\mu}_{n}=\int_{0}^{1}\varphi^{t}_{*}\left(\mu_{n}\right)dt
\end{equation} where $\varphi^{t}_{*}\left(\nu\right)$ is the  pushforward measure of $\nu$, i.e., $\varphi^{t}_{*}\left(\nu\right)\left(A\right):=\nu\left(\varphi^{-t}\left(A\right)\right).$ Up to a sub-sequence, we may assume that $\mu_n\to\mu$ for some $\varphi^{1}$-invariant measure $\mu$.

	\begin{itemize}
		\item By Theorem \ref{continuity of exponents for flow}, $\mu$ is also a MME. By the ergodic decomposition and the harmonic property of entropy, (almost) every ergodic component of $\mu$ is also a MME. As a consequence of  Ruelle's inequality, $\mu({\rm NUH}_{2\chi})=1$ for $\chi:=\frac{h_{\rm top}(X)}{3}$.
		\item By Theorem \ref{Pro:compare-two-sets}, given small $\varepsilon,\alpha>0$, we can find a small $\delta>0$ such that for any $C>1$, if $\mu_n(\cR_{C,\chi})>1-\delta$, then $\mu_{n}(\cR_{C,\chi, \varepsilon})>1-\alpha$.
		\item By Theorem~\ref{Thm:limit measure carries large exponent everywhere}, for such $\delta$ above, we can choose $C>1$ such that $\mu_n(\cR_{C,\chi})>1-\delta$ for all $n$ large enough. We observe that, due to the continuity of the derivative of the flow, there exists a constant $\widetilde{C} \geq C$ that depends only on the flow, ensuring that   
		\begin{equation}\label{R-C-chi is a subset}
			\cR_{C,\chi} \subset \varphi^{-t} \left(\cR_{\widetilde{C},\chi}\right), \quad \forall t \in [0,1].
		\end{equation} To see this, we consider the first inequality in the definition of $\cR_{C,\chi}$ as a representative example, since the remaining inequalities can be handled similarly.
		 For any $x\in\cR_{C,\chi}$, by definition, for any $n\in\mathbb{N}$, $$\left\|D\varphi^{-n}_{x}|_{E^{+}}\right\|\leq C{\rm e}^{-\chi n}.$$ Hence, 
			\begin{equation*}
			\begin{aligned}
			\left\|D\varphi^{-n}_{\varphi^{t}(x)}\big|_{E^{+}}\right\|
				&=\left\|\left(D\varphi^{t}_{\varphi^{-n}(x)}\circ D\varphi^{-n}_{x}\circ D\varphi^{-t}_{\varphi^{t}(x)}\right)\big|_{E^{+}}\right\|\\
				&\leq \text{CONST}\cdot\left\|D\varphi^{-n}_{x}|_{E^{+}}\right\|\\
				&\leq \widetilde{C}\cdot{\rm e}^{-\chi n}
			\end{aligned}
		\end{equation*}
for some $\widetilde{C}\geq C$ large enough. This shows $\varphi^{t}(x)\in\cR_{\widetilde{C},\chi}$ which gives $\cR_{C,\chi}\subset \varphi^{-t}\left(\cR_{\widetilde{C},\chi}\right).$ By (\ref{R-C-chi is a subset}), we then have $$\varphi^{t}_{*}\left(\mu_{n}\right)\left(\cR_{\widetilde{C},\chi}\right)\geq  \mu_n\left(\cR_{C,\chi}\right)>1-\delta,\quad \forall t\in[0,1]$$ which, by (\ref{ergodic component for time one map}), leads to $$\widetilde{\mu}_n\left(\cR_{\widetilde{C},\chi}\right)>1-\delta.$$
By Theorem \ref{Pro:compare-two-sets} with $\varepsilon,\alpha$ above, we finally have, for all large $n$, $$\widetilde{\mu}_n\left(\cR_{\widetilde{C},\chi,\varepsilon}\right)>1-\alpha.$$
	\end{itemize}

We cover $\cR_{\widetilde{C},\chi,\varepsilon}$ with finitely many balls ${B(x,\frac{\tau}{2})}$, where $\tau$ is the constant provided by Theorem \ref{Thm:Pesin}. By the finiteness of the cover, there exists at least one ball, say $B(x,\frac{\tau}{2})$, that carries positive measure for some distinct measures $\widetilde{\mu}_m$ and $\widetilde{\mu}_n$. This implies the existence of two sets $A, B \subset \cR_{\widetilde{C},\chi,\varepsilon}$ with $\widetilde{\mu}_m(A), \widetilde{\mu}_n(B) > 0$ such that the distance between $A$ and $B$ is less than $\tau$. By Theorem \ref{Thm:Pesin}, $\widetilde{\mu}_m$ and $\widetilde{\mu}_n$ are homoclinically related. Since both $\widetilde{\mu}_m$ and $\widetilde{\mu}_n$ are ergodic hyperbolic measures of maximal entropy, it follows from \cite[Corollary 1.2]{BCL25} (see also (\ref{Buzzi-Crovisier-Lima Corollary})) that $\widetilde{\mu}_m = \widetilde{\mu}_n$, leading to a contradiction.
\end{proof}

\section{Continuity of the Lyapunov exponents: work of Burguet revisited}\label{Burguet's work}
In this section, we will review several results from Burguet's work \cite{Bur24} and then prove Theorem \ref{invariant Cr case}. Recall the notations $\hat{f}$, $\hat{M}$, $\pi$, and others introduced in Section~\ref{tangent dynamics}.

\subsection{Entropy splitting into geometric component and neutral component}
Given an integer $L\in\mathbb{N}, \hat{H}\subset\hat{M}$ and a point $\hat{x}\in\hat{M}$, we define
$$E^{L}(\hat{x}):=\bigcup [a,b)$$ where the union runs over all $a,b\in\mathbb{Z}$ with $\hat{f}^{a}(\hat{x}), \hat{f}^{b}(\hat{x})\in\hat{H}$ and $0<b-a\leq L$. We can think of \( k \in E^{L}(\hat{x}) \) (or more commonly, $k$ with $\hat{f}^{k}(\hat{x})\in \hat{H}$) as a \emph{hyperbolic time} of \( \hat{x} \), even though, in this abstract case, the hyperbolicity in the usual sense is not directly involved. By convention, if $L=\infty$, we define $E^{\infty}(\hat{x}):=\mathbb{Z}$. Write $$E^{L}_{n}(\hat{x}):=E^{L}(\hat{x})\cap [0,n).$$ 

Let $G^{1} \subset \mathcal{R}$ denote the set of $x$ with $\dim E^{+}_{x} = 1$. Generated by these $E^{L}_{n}(\hat{x})$, we let $\mathcal{E}_{n}^{L}$ be the finite partition on $G^{1}$\footnote{The partition $\mathcal{E}_{n}^{L}$ is defined on $G^1$ rather than on the whole manifold for convenience. This causes no loss of generality for our purposes, since every ergodic measure with exactly one positive Lyapunov exponent gives full measure to $G^1$.} with atoms $$O_{E}:=\left\{x\in G^{1}:~E^{L}_{n}\left(\left(x, E^{+}_{x}\right)\right)=E\right\}$$ for $E\subset [0,n)$. Note that degenerate cases may occur. For instance, if $E^{L}_{n}(\hat{x}) = \emptyset$ for every $\hat{x} \in G^{1}$, then $\mathcal{E}_{n}^{L}$ reduces to the trivial partition consisting of a single element, i.e., $\mathcal{E}_{n}^{L} = \{G^{1}\}$. Given a finite partition $\mathcal{Q}$ on $M$ and a subset $E\subset\mathbb{N}$, we define the partition $$\mathcal{Q}^{E}:=\bigvee_{i\in E}f^{-i}\mathcal{Q}$$ and the finer partition (than $\mathcal{E}_{n}^{L}$) $$\mathcal{Q}^{\mathcal{E}_{n}^{L}}:=\bigcup_{O_{E}\in\mathcal{E}_{n}^{L}} \mathcal{Q}^{E}|_{O_{E}}$$ where $\mathcal{Q}|_{A}$ denotes the restriction on the subset $A$, i.e., $\mathcal{Q}|_{A}:=\{A\cap B:~ B\in \mathcal{Q}\}$. 

Recall that $W^{u}(x)$ denotes the usual Pesin unstable manifold at $x\in\mathcal{R}$. Given an invariant measure $\mu$, we say a measurable partition $\xi$ is \emph{subordinate to $W^{u}$} if for $\mu$-a.e. $x$, $\xi(x)\subset W^{u}(x)$ and  $\xi(x)$ contains an open neighborhood of $x$ w.r.t. the intrinsic topology on $W^{u}(x)$. We refer to \cite{LeS82} for the existence of subordinate partitions. Let $\{\mu_{x}\}$ be the conditional measures of $\mu$ w.r.t. a given measurable partition subordinate to $W^{u}$ (refer to \cite{Roh52} for background on measurable partitions and associated systems of conditional measures). 

Let $B_{n}(x,\delta)$ denote the $(n,\delta)$ dynamical ball at $x$, i.e., $$B_{n}(x,\delta):=\left\{y\in M:~ d\left(f^{i}(x), f^{i}(y)\right)<\delta,\,i=0,1,\cdots, n-1 \right\}.$$ Write $$H(t):=t\log \frac 1 t+(1-t)\log\frac 1{1-t},\,t\in(0,1).$$

\begin{Proposition}\label{entropy splitting}
	Let $f$ be a $C^{r}(r>1)$ diffeomorphism on a compact manifold $M$ without boundary and let $\mu$ be an ergodic measure with exactly one positive Lyapunov exponent. Let $\{\mu_{x}\}$ be a family of conditional measures of $\mu$ on the local unstable manifolds (w.r.t. some subordinate partition). Given any $\varepsilon>0$, let $\delta>0$ be sufficiently small so that one can choose a subset $\Lambda \subset M$ with $\mu(\Lambda) > \frac{1}{2}$, for which every $x \in \Lambda$ satisfies $$h(f,\mu)-\varepsilon\leq\liminf_{n\to+\infty}-\frac{1}{n}\log\mu_{x}\left(B_{n}\left(x,\delta\right)\right).$$
	
	Then for any large $L\in\mathbb{N}$, any subset $\hat{H}\subset\hat{M}$ (which generates $\mathcal{E}_{n}^{L}$) and any two finite partitions $\mathcal{P}, \mathcal{Q}$ on $M$ with $\diam(\mathcal{P})\leq\delta$, we have with $\mu_{x,\Lambda}:=\frac{\mu_{x}(\Lambda\cap\cdot)}{\mu_{x}(\Lambda)}$ $$h(f,\mu)-\varepsilon\leq\limsup_{n\to+\infty}\frac{1}{n}H_{\mu_{x,\Lambda}}\left(\mathcal{Q}^{\mathcal{E}_{n}^{L}}\bigg|\mathcal{E}_{n}^{L}\right)+ \limsup_{n\to+\infty}\frac{1}{n}H_{\mu_{x,\Lambda}}\left(\mathcal{P}^{[0,n)}\bigg|\mathcal{Q}^{\mathcal{E}_{n}^{L}}\right)+H\left(\frac{2}{L}\right),\quad \mu \text{ - a.e. } x\in\Lambda.$$
	
\end{Proposition}
\begin{Remark}
	By Ledrappier-Young entropy theory \cite{LeY85} (see also \cite{LuY26} and \cite{Zan22}), for any $\varepsilon > 0$, the set $\Lambda$ above can have $\mu$-measure arbitrarily close to one, provided that $\delta$ is chosen sufficiently small.
	
	Note that we impose no restrictions on the partition $\mathcal{Q}$, unlike $\mathcal{P}$, which is associated with the reference measure. In applying Proposition \ref{entropy splitting} to a sequence of ergodic measures, we aim to determine the entropy of the limit measure. To achieve this, the partition $\mathcal{Q}$ must be fixed in advance and remain independent of the measures under consideration (see the proof of Proposition \ref{invariant cr case with error term}).
	
\end{Remark}
\begin{proof}
	For $\mu$-a.e. $x\in\Lambda$, we may assume $\mu_{x}(\Lambda)>0$ so that $\mu_{x,\Lambda}:=\frac{\mu_{x}(\Lambda\cap\cdot)}{\mu_{x}(\Lambda)}$ is a well-defined probability measure. Given a finite partition $\mathcal{P}$ on $M$ with diameter less than $\delta$, we have 
	\begin{equation*}
		\begin{aligned}
			\liminf_{n\to+\infty}\frac{1}{n}H_{\mu_{x,\Lambda}}\left(\mathcal{P}^{[0,n)}\right)
			&=\liminf_{n\to+\infty}\frac{1}{n}\int-\log \mu_{x,\Lambda}\left(\mathcal{P}^{[0,n)}\left(y\right)\right)d\mu_{x,\Lambda}(y)\\
			&\geq\int\liminf_{n\to+\infty}-\frac{1}{n}\log \mu_{x,\Lambda}\left(\mathcal{P}^{[0,n)}\left(y\right)\right)d\mu_{x,\Lambda}(y)\\
			&\geq\int\liminf_{n\to+\infty}-\frac{1}{n}\log \mu_{x}\left(\mathcal{P}^{[0,n)}\left(y\right)\right)d\mu_{x,\Lambda}(y)\\
			&= \int\liminf_{n\to+\infty}-\frac{1}{n}\log \mu_{y}\left(\mathcal{P}^{[0,n)}\left(y\right)\right)d\mu_{x,\Lambda}(y)\\
			&\geq \int\liminf_{n\to+\infty}-\frac{1}{n}\log \mu_{y}\left(B_{n}\left(y,\delta\right)\right)d\mu_{x,\Lambda}(y)\\
			&\geq h(f,\mu)-\varepsilon.
		\end{aligned}
	\end{equation*}
	
	By definition, since the complement of \( E^{L}_{n}(\hat{x}) \) in $[0,n)$ is the disjoint union of intervals with length larger than \( L \) (except possibly the first and the last intervals), the number of possible types of \( E^{L}_{n}(\hat{x}) \) is at most \( \sum_{m=1}^{\lfloor 2n/L \rfloor + 10} \binom{n}{m} \), which implies that the cardinality of \( \mathcal{E}_{n}^{L} \) is also bounded by this sum. Recall the elementary property of conditional entropy (e.g., Theorem 4.3 in the book \cite{Wal82}): $$H_{\mu}\left(\xi\vee\eta|\zeta\right)=H_{\mu}\left(\xi|\zeta\right)+H_{\mu}\left(\eta|\xi\vee\zeta\right)\Longrightarrow H_{\mu}\left(\eta|\zeta\right)\leq H_{\mu}\left(\xi|\zeta\right)+H_{\mu}\left(\eta|\xi\right).$$ Given another finite partition $\mathcal{Q}$, we have
	
	\begin{equation*}
		\begin{aligned}
			H_{\mu_{x,\Lambda}}\left(\mathcal{P}^{[0,n)}\right)
			&\leq H_{\mu_{x,\Lambda}}\left(\mathcal{P}^{[0,n)}\big|\mathcal{E}_{n}^{L}\right)+H_{\mu_{x,\Lambda}}\left(\mathcal{E}_{n}^{L}\right)\\
			&\leq H_{\mu_{x,\Lambda}}\left(\mathcal{P}^{[0,n)}\big|\mathcal{E}_{n}^{L}\right)+\log\left(\sum_{m=1}^{\lfloor2n/L\rfloor+10}{n\choose m}\right)\\
			&\leq H_{\mu_{x,\Lambda}}\left(\mathcal{Q}^{\mathcal{E}_{n}^{L}}\bigg|\mathcal{E}_{n}^{L}\right)+ H_{\mu_{x,\Lambda}}\left(\mathcal{P}^{[0,n)}\bigg|\mathcal{Q}^{\mathcal{E}_{n}^{L}}\right)+\log\left(\sum_{m=1}^{\lfloor2n/L\rfloor+10}{n\choose m}\right).\\
		\end{aligned}
	\end{equation*}
	By Stirling's formula (see Lemma \ref{lemma:Stirling's formula} for a detailed explanation), we then have $$\liminf_{n\to+\infty}\frac{1}{n}H_{\mu_{x,\Lambda}}\left(\mathcal{P}^{[0,n)}\right)\leq \limsup_{n\to+\infty}\frac{1}{n}H_{\mu_{x,\Lambda}}\left(\mathcal{Q}^{\mathcal{E}_{n}^{L}}\bigg|\mathcal{E}_{n}^{L}\right)+ \limsup_{n\to+\infty}\frac{1}{n}H_{\mu_{x,\Lambda}}\left(\mathcal{P}^{[0,n)}\bigg|\mathcal{Q}^{\mathcal{E}_{n}^{L}}\right)+H\left(\frac{2}{L}\right)$$ which completes the proof.
\end{proof}

\subsection{Entropy estimation of geometric component}
Given a Borel probability measure $\mu$ and a finite subset $E\subset\mathbb{N}$, we write $$\mu^{E}:=\frac{1}{\# E}\sum_{i\in E}f_{*}^{i}\mu.$$

The  following result from \cite{Bur24} establishes a bound on the entropy of a partition over the set $E$ by relating it to the entropy of a measure averaged over the $E$-iterates of $f$ (i.e., $\mu^{E}$ above). Importantly, this result does not require the measure $\mu$ to be invariant, making it widely useful for entropy estimations in non-invariant settings.

\begin{Lemma}[{\cite[Lemma 8]{Bur24}}]\label{technical lemma}
	Let $(X,f)$ be a topological dynamical system and let $\mu$ be a Borel probability measure (not necessarily invariant). Given a finite subset $E\subset\mathbb{N}$ and a finite partition $\mathcal{Q}$, for any $m\in\mathbb{N}$, $$\frac{1}{\# E}H_{\mu}\left(\mathcal{Q}^{E}\right)\leq \frac{1}{m}H_{\mu^{E}}\left(\mathcal{Q}^{m}\right)+6m\frac{\#\left(\left(E+1\right)\Delta E\right)}{\# E}\log\#\mathcal{Q}$$ where $A\Delta B:=\left(A\setminus B\right)\cup \left(B\setminus A\right)$ and $E+1:=\{i+1:~i\in E\}$.
\end{Lemma}

Consider an ergodic measure $\mu$ with exactly one positive Lyapunov exponent ($\hat{\mu}^{+}$ denotes its ergodic unstable lift). Let $\chi^{L}$ be the indicator function of the subset $\left\{\hat{x}\in \hat{M}:~0\in E^{L}\left(\hat{x}\right)\right\}$ and let $$\beta^{L}:=\hat{\mu}^{+}\left(\left\{\hat{x}\in \hat{M}:~0\in E^{L}\left(\hat{x}\right)\right\}\right)=\int\chi^{L}d\hat{\mu}^{+}.$$

\begin{Proposition}\label{entropy estimation of geometric component}
	Let $f$ be a $C^{r}(r>1)$ diffeomorphism on a compact manifold $M$ without boundary and let $\mu$ be an ergodic measure with exactly one positive Lyapunov exponent. Let $\{\mu_{x}\}$ be a family of conditional measures of $\mu$ on the local unstable manifolds (w.r.t. some subordinate partition). Consider a subset $\hat{H}\subset\hat{M}$ (which generates $\mathcal{E}_{n}^{L}$). Let  $\Lambda\subset \mathcal{R}$ be a subset with $\mu(\Lambda)>\frac{1}{2}$ such that for any $L\in\mathbb{N}\cup\{\infty\}$, we have with $\hat{x}:=(x, E^{+}_{x})$,
	\begin{equation}\label{uniform convergence}
		\Big(\frac{1}{n}\sum_{i\in E^{L}_{n}(\hat{x})}\delta_{\hat{f}^{i}(\hat{x})}\Big)\xrightarrow{\text{weak $\ast$}} \chi^{L}\cdot\hat{\mu}^{+},\quad \text{uniformly in $x\in\Lambda$}.
	\end{equation}
	
	Then for any $L, m\in\mathbb{N}$ and any finite partition $\mathcal{Q}$ on $M$ with $\mu(\partial\mathcal{Q}^{m})=0$, we have with $\mu_{x,\Lambda}:=\frac{\mu_{x}(\Lambda\cap\cdot)}{\mu_{x}(\Lambda)}$  $$\limsup_{n\to+\infty}\frac{1}{n}H_{\mu_{x,\Lambda}}\left(\mathcal{Q}^{\mathcal{E}_{n}^{L}}\bigg|\mathcal{E}_{n}^{L}\right)\leq \beta^{L}\frac{1}{m}H_{(\beta^{L})^{-1}\chi^{L}\hat{\mu}^{+}}\left(\pi^{-1}\mathcal{Q}^{m}\right)+\frac{12m\log\#\mathcal{Q}}{L}\quad \mu \text{ - a.e. } x\in\Lambda$$ where $(\beta^{L})^{-1}\chi^{L}\hat{\mu}^{+}$ is the normalized measure (i.e., probability measure) of $\chi^{L}\hat{\mu}^{+}$ whenever $\beta^{L}\neq 0$. In the degenerate case $\beta^{L} = 0$, the left-hand side of the above inequality is zero, so the conclusion holds trivially.
\end{Proposition}
\begin{Remark}\label{remark of entropy estimation of geometric component}
	Note that by definition, $$i\in E^{L}\left(\hat{x}\right)\Longleftrightarrow\hat{f}^{i}(\hat{x})\in \left\{\hat{x}\in \hat{M}:~0\in E^{L}\left(\hat{x}\right)\right\}\Longleftrightarrow \chi^{L}(\hat{f}^{i}(\hat{x}))=1.$$ Hence the ergodicity of $\hat{\mu}^{+}$ implies that for $\hat{\mu}^{+}$-a.e. $\hat{x}$, $$\Big(\frac{1}{n}\sum_{i\in E^{L}_{n}(\hat{x})}\delta_{\hat{f}^{i}(\hat{x})}\Big)\xrightarrow{\text{weak $\ast$}} \chi^{L}\cdot\hat{\mu}^{+},\quad n\to+\infty.$$ Therefore it is reasonable to assume the set $\Lambda$ above has measure larger than $\frac{1}{2}$.
\end{Remark}
\begin{proof}
	Note that for $\mu$-a.e. $x\in\Lambda$, we may assume $\mu_{x}(\Lambda)>0$ so that $\mu_{x,\Lambda}:=\frac{\mu_{x}(\Lambda\cap\cdot)}{\mu_{x}(\Lambda)}$ is a well-defined probability measure.
	
	Let us first consider the degenerate case $\beta^{L} = 0$, which is equivalent to $E^{L}(\hat{x}) = \emptyset$ for $\hat{\mu}^{+}$-a.e. $\hat{x}$. By definition, we have the partition $\mathcal{E}_{n}^{L} \overset{\circ}{=} \{G^{1}\}$ up to a set of measure zero, and consequently, $\mathcal{Q}^{\mathcal{E}_{n}^{L}} \overset{\circ}{=} \mathcal{Q}$. Hence, for $\mu$-a.e. $x \in \Lambda$, $$\limsup_{n\to+\infty}\frac{1}{n}H_{\mu_{x,\Lambda}}\left(\mathcal{Q}^{\mathcal{E}_{n}^{L}}\bigg|\mathcal{E}_{n}^{L}\right)=0.$$
	
	 It thus remains to consider the general case $\beta^{L} \neq 0$. Given a subset $O_{E}\in \mathcal{E}_{n}^{L}$, recall that $$\mu_{x,\Lambda\cap O_{E}}:=\frac{\mu_{x,\Lambda}( O_{E}\cap\cdot)}{\mu_{x,\Lambda}(O_{E})}=\frac{\mu_{x}(\Lambda\cap O_{E}\cap\cdot)}{\mu_{x}(\Lambda\cap O_{E})}$$ whenever $\mu_{x,\Lambda}(O_{E})>0$. 
	We define the unstable lift $\hat{\mu}^{+}_{x,\Lambda\cap O_{E}}$ of $\mu_{x,\Lambda\cap O_{E}}$ by $$\hat{\mu}^{+}_{x,\Lambda\cap O_{E}}:=\int_{\hat{M}}\delta_{(y, E^{+}_{y})} d\left(\mu_{x,\Lambda\cap O_{E}}\right)\left(y\right).$$

	Note by definition (write $\hat{y}:=(y, E^{+}_{y})$), $$\left(\hat{\mu}_{x,\Lambda\cap O_{E}}^{+}\right)^{E}=\frac{n}{\# E}\int\Big(\frac{1}{n}\sum_{i\in E^{L}_{n}(\hat{y})=E}\delta_{\hat{f}^{i}(\hat{y})}\Big)d\left(\mu_{x,\Lambda\cap O_{E}}\right)\left(y\right).$$ As a consequence, by (\ref{uniform convergence}), $$\frac{\#E^{L}_{n}(\hat{x})}{n}\rightarrow \beta^{L}\quad\text{uniformly in $x\in\Lambda$}$$ and 
	$$\left(\hat{\mu}_{x,\Lambda\cap O_{E}}^{+}\right)^{E}\xrightarrow{\text{weak $\ast$}} (\beta^{L})^{-1}\chi^{L}\cdot\hat{\mu}^{+}\quad\text{uniformly in $x\in\Lambda$}.$$
	Together with the assumption $\hat{\mu}^{+}\left(\partial\left(\pi^{-1}\mathcal{Q}^{m}\right)\right)=0$ (since $\mu(\partial\mathcal{Q}^{m})=0$), by Lemma \ref{technical lemma}, we have 
	\begin{equation*}
		\begin{aligned}
			\frac{1}{n}H_{\mu_{x,\Lambda}}\left(\mathcal{Q}^{\mathcal{E}_{n}^{L}}\bigg|\mathcal{E}_{n}^{L}\right)
			&=\frac{1}{n}\sum_{E\subset [0,n)}\mu_{x,\Lambda}\left(O_{E}\right) \cdot H_{\mu_{x,\Lambda\cap O_{E}}}\left(\mathcal{Q}^{E}\right)\quad\text{(by definition)}\\
			&\leq \sum_{E\subset [0,n)}\mu_{x,\Lambda}\left(O_{E}\right)\cdot\frac{\# E}{n}\cdot\left(\frac{1}{m}H_{\left(\mu_{x,\Lambda\cap O_{E}}\right)^{E}}\left(\mathcal{Q}^{m}\right)+6m\frac{\#\left(\left(E+1\right)\Delta E\right)}{\# E}\log\#\mathcal{Q}\right)\\
			&=\sum_{E\subset [0,n)}\mu_{x,\Lambda}\left(O_{E}\right)\cdot\frac{\# E}{n}\cdot\left(\frac{1}{m}H_{\left(\hat{\mu}_{x,\Lambda\cap O_{E}}^{+}\right)^{E}}\left(\pi^{-1}\mathcal{Q}^{m}\right)+6m\frac{\#\left(\left(E+1\right)\Delta E\right)}{\# E}\log\#\mathcal{Q}\right)\\
			\quad\text{ as } n\to+\infty\quad&\leq\beta^{L}\frac{1}{m}H_{(\beta^{L})^{-1}\chi^{L}\hat{\mu}^{+}}\left(\pi^{-1}\mathcal{Q}^{m}\right)+\frac{12m\log\#\mathcal{Q}}{L}.\\
		\end{aligned}
	\end{equation*}
In the final step, we use the fact that $\#\left((E+1) \Delta E\right) \leq \frac{2n}{L}+10$, which follows from the observation that the complement of $E$ is a disjoint union of intervals, each with length larger than $L$ (except possibly the first and the last intervals).

\end{proof}

\subsection{Entropy estimation of neutral component}\label{Section Entropy estimation of neutral component}

Recall that $G^{1} \subset \mathcal{R}$ denotes the set of $x$ with $\dim E^{+}_{x} = 1$. Given $\kappa > 0$, let $G^{1}_{\kappa}\subset G^{1}$ denote the set of points $x$ such that $\exp_{x}^{-1}\left(W^{u}(x)\right)$ contains the graph of a 1-Lipschitz map $h: I_{\kappa} \to \left(E^{+}_{x}\right)^{\bot}$. Here, $I_{\kappa}$ denotes the interval in $E^{+}_{x}$ centered at 0 with length $\kappa$, and $\left(E^{+}_{x}\right)^{\bot}$ represents the subspace in $T_{x}M$ orthogonal to $E^{+}_{x}$. $G_{\kappa}^{1}$ is a Borel measurable set (see \cite[Lemma 4]{Bur24}). Roughly speaking, $G_{\kappa}^{1}$ is the set of points $x$ whose Pesin unstable manifold is one-dimensional and has length at least $\kappa$. 

Recall that $\rho:\hat{M}\to\mathbb{R}$ is the continuous function $$\rho(x, E):=\log\left\|Df_{x}|_{E}\right\|.$$ Define another continuous function $\rho':\hat{M}\to\mathbb{R}$ (introduced by Burguet in \cite{Bur24}) by $$\rho'(x,E):=\rho(x, E)-\frac{1}{r}\log^{+}\left\|Df_{x}\right\|.$$  Given $\delta>0$, let $H_{\delta} \subset \hat{M}$ be the set of $\hat{x}=(x,E)$ such that for any $n>0$,  $$\sum_{i=0}^{n-1}\rho'\left(\hat{f}^{i-n}(\hat{x})\right)=\sum_{i=0}^{n-1}\rho\left(\hat{f}^{i-n}(\hat{x})\right)-\frac{1}{r}\sum_{i=0}^{n-1}\log^{+}\left\|Df_{f^{i-n}(x)}\right\|\geq n\delta.$$ By its definition, we can expect (see \cite[Lemma 3]{Bur24}) that for an ergodic measure $\mu$ with exactly one positive exponent satisfying $$\lambda^{+}(f,\mu)>\frac{\log\left\|Df\right\|}{r}+\delta,$$ we have $$\hat{\mu}^{+}(H_{\delta})>0.$$ We can interpret the set $H_{\delta}$, which generates the so-called "hyperbolic times", in the usual manner: an integer $n$ is referred to as a \emph{hyperbolic time} for $\hat{x}$ if $\hat{f}^{n}(\hat{x}) \in H_{\delta}$. The key distinction from Buzzi-Crovisier-Sarig \cite{BCS22} lies in the function $\rho'$, which quantifies the degree of non-uniform hyperbolicity and defines $H_{\delta}$, whereas in \cite{BCS22}, the corresponding function is $\rho$. We will not explicitly work with hyperbolic times, as their role is contained in the proof of \cite[Proposition 1]{Bur24} (see Proposition \ref{Burguet neutral estimation} below).

Recall $H(t):=t\log \frac 1 t+(1-t)\log\frac 1{1-t},\,t\in(0,1)$ and write $$\delta^{L}_{n}(\hat{x}):=\frac{1}{n}\sum_{i\in E^{L}_{n}(\hat{x})}\delta_{\hat{f}^{i}(\hat{x})}.$$

Next, we revisit a result from Burguet \cite{Bur24} which, roughly speaking, estimates the maximal number of reparametrizations required to locally cover a small segment of a Pesin unstable manifold. This is where the special function $\rho'$ and the set $H_{\delta}$ defined by $\rho'$ play a crucial role.

\begin{Proposition}[{\cite[Proposition 1]{Bur24}}]\label{Burguet neutral estimation}
	Let $f$ be a $C^{r}(r>1)$ diffeomorphism on a compact manifold $M$ without boundary. There are $\kappa>0$ and a sequence $\{\delta_{q}>0\}_{q\in\mathbb{N}}$ with the following property.
	
	Given $\delta>0$, consider $\hat{H}:=\pi^{-1}G_{\kappa}^{1}\cap H_{\delta}$ (which generates $\mathcal{E}_{n}^{L}$). For any $L, n, q\in\mathbb{N}$, any finite partition $\mathcal{Q}$ on $M$ with diameter less than $\delta_{q}$, any unstable curve $\sigma\subset W^{u}_{\rm loc}(x)$ (for some $x\in G^{1}$, i.e., $\dim E^{+}_{x} = 1$) and any atom $A\in\mathcal{Q}^{\mathcal{E}_{n}^{L}}$, there is a family $\Phi_{A\cap\sigma}$ of $C^{r}$ curves $\gamma:[-1,1]\to M$ such that 
	\begin{itemize}
		\item $A\cap\sigma\subset \cup_{\gamma\in \Phi_{A\cap\sigma}}\gamma([-1,1])$,
		\item $\left\|d\left(f^{k}\circ \gamma\right)\right\|\leq 1$ for $0\leq k\leq n-1$ and $\gamma\in \Phi_{A\cap\sigma}$,
		\item writing $\hat{x}:=(x, E^{+}_{x})$ and $E^{L}_{n}:=E^{L}_{n}(\hat{x})$ (since $E^{L}_{n}(\hat{x})$ is identical for all $x\in A\cap\sigma$),
		\begin{equation*}
			\begin{aligned}
				 \frac{1}{n}\log\#\Phi_{A\cap \sigma}
				&\leq \left(1-\frac{\#E^{L}_{n}}{n}\right)\cdot \left(\frac{\log\|Df\|}{r}+2A_{f}H(A_{f}^{-1})+B_{r}\right)\\
				&+\left(\log 2 +\frac{1}{r-1}\right)\cdot \left(\int\frac{\log^{+}\left\|Df^{q}_{y}\right\|}{q}d\left(\pi_{*}\left(\delta^{L}_{n}\left(\hat{x}\right)\right)\right)\left(y\right)-\int\rho d\left(\delta^{L}_{n}\left(\hat{x}\right)\right)\right)\\
				&+ \tau_{1}(L,q)+\tau_{2}(n)
			\end{aligned}
		\end{equation*} where $A_{f}:=\log\left\|Df\right\|+\log\left\|Df^{-1}\right\|+1$ and $B_{r}$ is a universal constant depending only on $r$. The error terms $\tau_{1}(L,q), \tau_{2}(n)$ satisfy $$\limsup_{q\to+\infty}\limsup_{L\to+\infty}\tau_{1}(L,q)=0,\quad\lim_{n\to+\infty}\tau_{2}(n)=0.$$
	\end{itemize}
\end{Proposition}
\begin{Remark} \noindent
	\begin{enumerate}
		\item The primary term of interest is $\frac{\log\|Df\|}{r}$ which, by applying large iterations of $f$, generates the commonly observed threshold $\frac{\lambda^{+}(f)}{r}$ in Yomdin theory \cite{Yom87}. In Yomdin theory, the term $\frac{\lambda^{+}(f)}{r}$ typically represents the maximal number (approximately $e^{n \cdot \frac{\lambda^{+}(f)}{r}}$) of reparametrizations  required to cover the $f^{n}$-image of a curve within a dynamical ball $B_{n}(x, \varepsilon)$.
		\item We typically refer to \( k \in E^{L}(\hat{x}) \) as a \emph{hyperbolic time} of \( \hat{x} \), even though, in this abstract case, the hyperbolicity in the usual sense is not directly involved. We may think of $\mathcal{Q}^{\mathcal{E}_{n}^{L}}$ as a partition (according to the hyperbolic times) of some Pesin's local unstable manifold. To illustrate the idea, let us simply assume that we are working in the uniformly hyperbolic case (say an Anosov automorphism $f$ on $\mathbb{T}^2$ with an exponent $\lambda>0$).  Any piece $A\cap\sigma$ ($A$ represents an element in $\mathcal{Q}^{\mathcal{E}_{n}^{L}}$) of a local unstable manifold $\sigma$ inside $B_{n}(x,\varepsilon)$ has length $\lesssim e^{-n\lambda}$ and hence can be reparameterized by some $\gamma:[-1,1]\to M$ with $\|d\,\gamma\| \leq e^{-n\lambda}$. Consequently, thanks to the uniform expansion on the local unstable manifolds, we have $\|d(f^{k}\circ \gamma)\| \leq 1$ for $0\leq k\leq n-1$. This means that one single curve $\gamma$ is enough to properly cover $A\cap\sigma$.  We now return to the non-uniformly hyperbolic setting. The family $\Phi_{A\cap\sigma}$ of $C^{r}$ curves $\gamma:[-1,1]\to M$ has to be constructed inductively based on whether the time is hyperbolic. At a hyperbolic time $k$ (i.e., $k\in E^{L}_{n}$), one can roughly think that the uniform expansion (i.e., backward contraction) can be somehow recovered. The condition $\|d(f^{k}\circ \gamma)\| \leq 1$ is automatically satisfied for $\gamma$ constructed from the previous inductive step up to some negligible error indicated by the subsequent term  $\left(\int\cdots-\int\cdots\right)$ (see Section 4.2 in \cite{Bur24} for precise arguments) which tends to zero for appropriately chosen points $\hat{x}$ as $q\to +\infty$ (see the proof of Proposition \ref{invariant cr case with error term}). However, at a non-hyperbolic time $k$ (i.e., $k\in [0, n)\setminus E^{L}_{n}$), since we do not have any hyperbolicity, we need to divide  $\gamma$  into multiple smaller segments (reparametrizations) $ \theta_{i} $ to ensure that the condition $ \|d(f^{k} \circ \gamma \circ \theta_{i})\|  \leq 1 $ is once again satisfied. The cardinality of these smaller segments is bounded above by the term $e^{\frac{\log\|Df\|}{r}}$ in Yomdin theory and the number of non-hyperbolic times is approximately $n-E^{L}_{n}$. This results in the main term $\left(1-\frac{\#E^{L}_{n}}{n}\right)\cdot \frac{\log\|Df\|}{r}$ in the conclusion. 
		\item The original statement of {\cite[Proposition 1]{Bur24}} is presented in the context of surface diffeomorphisms. However, the proof primarily addresses the covering of one-dimensional curves and can be readily adapted to higher-dimensional systems, provided we focus on the dynamics restricted to curves. Another distinction lies in our use of the empirical measure $\delta^{L}_{n}(\hat{x})$ instead of its average over $A$ w.r.t. some measure (denoted as $\hat{\zeta}^{M}_{F_{n}}$ in their terminology) defined on the unstable leaf. This approach is justified by the fact that the  diameter of $f^{i}(A)$ is less than the small constant $\delta_{q}$ for all $0 \leq i \leq n-1$, ensuring that any two points in $A$ admit almost identical empirical measures. Consequently, there is no significant difference between $\delta^{L}_{n}(\hat{x})$ and its average over $A$. For further details, see Section 4.2 in \cite{Bur24}. The motivation for using $\delta^{L}_{n}(\hat{x})$ is that it does not rely on a predefined reference measure, unlike $\hat{\zeta}^{M}_{F_{n}}$ in \cite{Bur24}, which is defined on a subset relative to such a measure (although this dependence is not used in the proof).
		
		\item The parameters $\kappa$ and $\{\delta_{q}\}_{q \in \mathbb{N}}$ indicate the scale at which a special variant of Yomdin's theory, introduced by Burguet (see \cite[Lemma 12]{Bur24}), can be applied.

	\end{enumerate}
	
\end{Remark}

Recall that $\chi^{L}$ denotes the indicator function of the subset $\left\{\hat{x}\in \hat{M}:~0\in E^{L}\left(\hat{x}\right)\right\}$ and $\beta^{L}:=\hat{\mu}^{+}\left(\left\{\hat{x}\in \hat{M}:~0\in E^{L}\left(\hat{x}\right)\right\}\right)$.

The following result is a direct application of Proposition \ref{Burguet neutral estimation}, offering an entropy upper bound for the neutral component. Recall that $A_{f}:=\log\left\|Df\right\|+\log\left\|Df^{-1}\right\|+1$ and $B_{r}$ is a universal constant depending only on $r$ from Proposition \ref{Burguet neutral estimation}.
\begin{Proposition}\label{entropy estimation of neutral component}
	Let $f$ be a $C^{r}(r>1)$ diffeomorphism on a compact manifold $M$ without boundary. There are $\kappa>0$ and a sequence $\{\delta_{q}>0\}_{q\in\mathbb{N}}$ with the following property.
	
	Let $\mu$ be an ergodic measure with exactly one positive Lyapunov exponent and let $\{\mu_{x}\}$ be a family of conditional measures of $\mu$ on the local unstable manifolds (w.r.t. some subordinate partition $\xi^{u}$). 	Given $\delta>0$, consider $\hat{H}:=\pi^{-1}G_{\kappa}^{1}\cap H_{\delta}$ (which generates $\mathcal{E}_{n}^{L}$). Let  $\Lambda\subset \mathcal{R}$ be a subset with $\mu(\Lambda)>\frac{1}{2}$ such that for any $L\in\mathbb{N}\cup\{\infty\}$, we have with $\hat{x}:=(x, E^{+}_{x})$,
	\begin{equation}\label{assumption: uniform convergence}
		\Big(\frac{1}{n}\sum_{i\in E^{L}_{n}(\hat{x})}\delta_{\hat{f}^{i}(\hat{x})}\Big)\xrightarrow{\text{weak $\ast$}} \chi^{L}\cdot\hat{\mu}^{+},\quad \text{uniformly in $x\in\Lambda$}.
	\end{equation}
	Then for any $L, q\in\mathbb{N}$ and any two finite partitions $\mathcal{P}, \mathcal{Q}$ on $M$ with $\mu(\partial\mathcal{P})=0$ and $\diam(\mathcal{Q})\leq\delta_{q}$,
	{\small \begin{equation*}
		\begin{aligned}
			\limsup_{n\to+\infty}\frac{1}{n}H_{\mu_{x,\Lambda}}\left(\mathcal{P}^{[0,n)}\bigg|\mathcal{Q}^{\mathcal{E}_{n}^{L}}\right)
			&\leq \left(1-\beta^{L}\right)\cdot \left(\frac{\log\left\|Df\right\|}{r}+2A_{f}H(A_{f}^{-1})+B_{r}\right)\\
			&+\left(\log 2 +\frac{1}{r-1}\right)\cdot \left(\int\frac{\log^{+}\left\|Df^{q}_{x}\right\|}{q}d\left(\pi_{*}\left(\chi^{L}\hat{\mu}^{+}\right)\right)\left(x\right)-\int\rho d\left(\chi^{L}\hat{\mu}^{+}\right)\right)\\
			&+ \tau_{1}(L,q),\quad \mu\text{ - a.e. } x\in\Lambda
		\end{aligned}
	\end{equation*}}
	where the error term $\tau_{1}(L,q)$ satisfies $$\limsup_{q\to+\infty}\limsup_{L\to+\infty}\tau_{1}(L,q)=0.$$
\end{Proposition}
	\begin{proof}
		Note that as before, for $\mu$-a.e. $x\in\Lambda$, we may assume $\mu_{x}(\Lambda)>0$ so that $\mu_{x,\Lambda}:=\frac{\mu_{x}(\Lambda\cap\cdot)}{\mu_{x}(\Lambda)}$ is a well-defined probability measure.

	The constant $\kappa>0$ and the sequence $\{\delta_{q}>0\}_{q\in\mathbb{N}}$ are from Proposition \ref{Burguet neutral estimation}. Given any unstable curve $\sigma\subset W^{u}_{\rm loc}(x)$ (for some $\mu$-typical point $x\in\Lambda$) and $A\in \mathcal{Q}^{\mathcal{E}_{n}^{L}}$, let  $\Phi_{A\cap \sigma}$ be the family of $C^{r}$ curves $\gamma:[-1,1]\to M$ that covers $A\cap\sigma$ with the properties listed in Proposition \ref{Burguet neutral estimation}. Since our primary concern is the entropy estimation on $\Lambda$ with respect to the measure $\mu_{x,\Lambda}$, we may, after possibly restricting to a sub-family, assume that each $\gamma \in \Phi_{A \cap \sigma}$ has a non-empty intersection with $A \cap \sigma \cap \Lambda$.
	
The following result states that there are at most sub-exponentially many elements in $\mathcal{P}^{[0,n)}$ intersecting $\Lambda\cap\gamma $.
		\begin{Claim}
			$$\limsup_{n\to+\infty}\frac{1}{n}\log\sup_{A\in\mathcal{Q}^{\mathcal{E}_{n}^{L}}, \gamma\in\Phi_{A\cap\sigma}}\#\left\{B\in\mathcal{P}^{[0,n)}:~ B\cap \Lambda\cap\gamma\neq \emptyset\right\}=0.$$
		\end{Claim}
		\begin{proof}[Proof of Claim]
			Since $\mu(\partial\mathcal{P})=0$, given $\varepsilon>0$, we can find $\tau>0$ small enough and a continuous function $h: M\to (0,+\infty)$ such that 
			\begin{itemize}
				\item $h(x)=1$ if $x$ is in the $\tau$-neighborhood of $\partial\mathcal{P}$ (denoted by $B(\partial\mathcal{P}, \tau)$),
				\item $\int h d\mu<\varepsilon$.
			\end{itemize} 
			By the assumption (\ref{assumption: uniform convergence}) for the case $L=\infty$, we have with $\hat{x}:=(x, E^{+}_{x})$ $$\Big(\frac{1}{n}\sum_{i\in [0,n)}\delta_{\hat{f}^{i}(\hat{x})}\Big)\xrightarrow{\text{weak $\ast$}} \hat{\mu}^{+},\quad \text{uniformly in $x\in\Lambda$}.$$ Consequently, $$\Big(\frac{1}{n}\sum_{i\in [0,n)}\delta_{f^{i}(x)}\Big)\xrightarrow{\text{weak $\ast$}} \mu,\quad \text{uniformly in $x\in\Lambda$}.$$ In particular for the function $h$, we have uniformly for all $x\in\Lambda$, \begin{equation}\label{consequence of uniform convergence}
			\limsup_{n\to+\infty}\frac{1}{n}\#\left\{k\in[0,n):~f^{k}(x)\in B(\partial\mathcal{P}, \tau)\right\}\leq \lim_{n\to+\infty}\frac{1}{n}\sum_{k\in [0,n)}h(f^{k}(x))=\int h d\mu<\varepsilon. 
			\end{equation}
			To prove the claim, it is sufficient to prove for all large $n$, $$\sup_{A\in\mathcal{Q}^{\mathcal{E}_{n}^{L}}, \gamma\in\Phi_{A\cap\sigma}}\#\left\{B\in\mathcal{P}^{[0,n)}:~ B\cap \Lambda\cap\gamma\neq \emptyset\right\}\leq (2\tau^{-1}+1)\left(\#\mathcal{P}\right)^{n\varepsilon}.$$ Going by contradiction, assuming for a sufficiently large $n$, there are some $A\in\mathcal{Q}^{\mathcal{E}_{n}^{L}}$ and $ \gamma\in\Phi_{A\cap\sigma}$ violating the above inequality, i.e., $$\#\left\{B\in\mathcal{P}^{[0,n)}:~ B\cap \Lambda\cap\gamma\neq \emptyset\right\}> (2\tau^{-1}+1)\left(\#\mathcal{P}\right)^{n\varepsilon}.$$ Since $\|d(f^{k}\circ \gamma)\| \leq 1$ for $0\leq k\leq n-1$, we may reparametrize $\gamma$ with at most $(2\tau^{-1}+1)$ affine contractions $\theta: [-1,1]\to [-1,1]$ such that $f^{k}\circ\gamma\circ\theta$ has length less than $\tau$ for $0\leq k\leq n-1$. We may assume that every such $\theta$ satisfies $\Lambda\cap\gamma\circ\theta\neq \emptyset$. Note that among these contractions, there must be some $\theta$ such that 
			\begin{equation}\label{many intersections}
				\#\Omega_{1}:=\left\{B\in\mathcal{P}^{[0,n)}:~ B\cap \Lambda\cap\gamma\circ\theta\neq \emptyset\right\}> \left(\#\mathcal{P}\right)^{n\varepsilon}.
			\end{equation}

			As a consequence, we have $$\#\Omega_{2}:=\left\{k\in[0,n):~\partial\mathcal{P}\cap f^{k}\circ\gamma\circ\theta  \neq \emptyset\right\}\geq n\varepsilon.$$ Indeed, if $k\notin \Omega_{2}$, then $f^{k}\circ\gamma\circ\theta$ is contained in a single atom of $\mathcal P$. Consequently, for every $B\in\Omega_{1}\subset\mathcal P^{[0,n)}$, the set $f^{k}(B)$ must lie in that same atom. In other words, if $k\notin \Omega_{2}$, then the $k$-th symbol of any itinerary corresponding to an element of $\Omega_{1}$ is uniquely determined. Therefore, $$\#\Omega_{1}\leq (\#\mathcal P)^{\#\Omega_{2}}.
			$$ Together with (\ref{many intersections}), this implies that $\#\Omega_{2}\ge n\varepsilon$. Recall $f^{k}\circ\gamma\circ\theta$ has length less than $\tau$ for all $k\in[0,n)$. Hence for any $x\in \Lambda\cap\gamma\circ\theta$, $$\#\left\{k\in[0,n):~ f^{k}(x)\in B(\partial\mathcal{P}, \tau)\right\}\geq \#\Omega_{2}\geq n\varepsilon$$ which contradicts the uniform convergence in  (\ref{consequence of uniform convergence}).					
		\end{proof}
	We note that the element $\xi^{u}(x)$ of the subordinate partition $\xi^{u}$ which supports $\mu_{x,\Lambda}$ can be covered by finitely many local unstable curves $\{\sigma_{i}\}_{1\leq i\leq l}$. As a consequence of the claim above, by Proposition \ref{Burguet neutral estimation}, we have 
		\begin{equation*}
			\begin{aligned}
				\limsup_{n\to+\infty}\frac{1}{n}H_{\mu_{x,\Lambda}}\left(\mathcal{P}^{[0,n)}\bigg|\mathcal{Q}^{\mathcal{E}_{n}^{L}}\right)
				&\leq\limsup_{n\to+\infty}\frac{1}{n}\sum_{A\in \mathcal{Q}^{\mathcal{E}_{n}^{L}}}\mu_{x,\Lambda}(A)\cdot \log\#\left\{B\in\mathcal{P}^{[0,n)}:~ B\cap\Lambda\cap A\cap\xi^{u}(x)\neq \emptyset\right\}\\
				&=\limsup_{n\to+\infty}\frac{1}{n}\sum_{A\in \mathcal{Q}^{\mathcal{E}_{n}^{L}}}\mu_{x,\Lambda}(A)\cdot\log\max_{i}\#\left\{B\in\mathcal{P}^{[0,n)}:~ B\cap\Lambda\cap A\cap\sigma_{i}\neq \emptyset\right\}\\
				&\leq\limsup_{n\to+\infty}\frac{1}{n}\sum_{A\in \mathcal{Q}^{\mathcal{E}_{n}^{L}}}\mu_{x,\Lambda}(A)\cdot\log\max_{i}\#\Phi_{A\cap\sigma_{i}} \quad\text{ ( by Claim  ) }\\
				&\leq \left(1-\beta^{L}\right)\cdot \left(\frac{\log\left\|Df\right\|}{r}+2A_{f}H(A_{f}^{-1})+B_{r}\right)\\
				&+\left(\log 2+\frac{1}{r-1}\right)\cdot \left(\int\frac{\log^{+}\left\|Df^{q}_{x}\right\|}{q}d\left(\pi_{*}\left(\chi^{L}\hat{\mu}^{+}\right)\right)\left(x\right)-\int\rho d\left(\chi^{L}\hat{\mu}^{+}\right)\right)\\
				&+ \tau_{1}(L,q),\quad \mu\text{ - a.e. } x\in\Lambda.
			\end{aligned}
		\end{equation*}	
		Note that in the last step, we have used the following uniform convergence for $x\in\Lambda$, writing $\hat{x}:=(x, E^{+}_{x})$, $$\frac{\#E^{L}_{n}(\hat{x})}{n}\rightarrow \beta^{L},\quad\delta^{L}_{n}\left(\hat{x}\right):=\Big(\frac{1}{n}\sum_{i\in E^{L}_{n}(\hat{x})}\delta_{\hat{f}^{i}(\hat{x})}\Big)\xrightarrow{\text{weak $\ast$}} \chi^{L}\cdot\hat{\mu}^{+}.$$

	\end{proof}

\subsection{Proof of Theorem \ref{invariant Cr case}}
For convenience, we denote by
$$
C(f):=\frac{\log\left\|Df\right\|}{r}+2A_{f}H(A_{f}^{-1})+B_{r}
$$
the quantity appearing in Proposition~\ref{Burguet neutral estimation} and Proposition~\ref{entropy estimation of neutral component},
where $$A_{f}:=\log\|Df\|+\log\|Df^{-1}\|+1$$ and $B_{r}$ is a universal constant depending only on $r$.

We now proceed to prove the following result, which closely resembles Theorem \ref{invariant Cr case}, with the only difference being that $\alpha$ is replaced by $C(f)$. The idea is borrowed from Burguet \cite{Bur24}.
\begin{Proposition}\label{invariant cr case with error term}
		Let $f$ be a $C^{r}(r>1)$ diffeomorphism on a compact manifold $M$ without boundary. Let $\{\nu_{k}\}_{k\geq 1}$ be a sequence of ergodic measures such that each $\nu_{k}$ has exactly one positive Lyapunov exponent. Assume that the corresponding unstable lift $\hat{\nu}_{k}^{+}$ converges to some invariant measure $\hat{\mu}$ whose projection $\mu$ also has exactly one positive Lyapunov exponent and $\lim_{k\to+\infty} \lambda^{+}(f,\nu_{k})>\frac{\log\|Df\|}{r}$.

	Then there is a decomposition of $\hat{\mu}=(1-\beta)\hat{\mu}_{0}+\beta \hat{\mu}^{+}_{1}$ for some $\beta\in [0,1]$ and some $\hat{f}$-invariant measures $\hat{\mu}_{0},\hat{\mu}_{1}^{+}$ ($\hat{\mu}_{1}^{+}$ is the unstable lift of some $f$-invariant measure $\mu_{1}$) such that $$\limsup_{k\to+\infty}h(f,\nu_{k})\leq\beta h(f,\mu_{1})+(1-\beta)\cdot C(f).$$
\end{Proposition}

\begin{proof}
	Let $\kappa>0$ and  $\{\delta_{q}>0\}_{q\in\mathbb{N}}$ be the constants in Proposition \ref{entropy estimation of neutral component}. Let $\delta>0$ be small enough such that 
	\begin{equation}\label{limit measure exponent larger than delta}
		\lim_{k\to+\infty} \lambda^{+}(f,\nu_{k})>\frac{\log\left\|Df\right\|}{r}+\delta.
	\end{equation}
	We fix $\hat{H}:=\pi^{-1}G_{\kappa}^{1}\cap H_{\delta}.$ Recall that $\chi^{L}$ denotes the indicator function of $\left\{\hat{x}\in \hat{M}:~0\in E^{L}\left(\hat{x}\right)\right\}$. Note (see Remark \ref{remark of entropy estimation of geometric component}) that the ergodicity of $\hat{\nu}_{k}^{+}$ implies that for $\hat{\nu}_{k}^{+}$-a.e. $\hat{x}$, 
	\begin{equation}\label{weak star convergence for typical points}
		\Big(\frac{1}{n}\sum_{i\in E^{L}_{n}(\hat{x})}\delta_{\hat{f}^{i}(\hat{x})}\Big)\xrightarrow{\text{weak $\ast$}} \chi^{L}\cdot\hat{\nu}_{k}^{+},\quad n\to+\infty.
	\end{equation}
	Write $\beta^{L}_{k}:=\hat{\nu}_{k}^{+}\left(\left\{\hat{x}\in \hat{M}:~0\in E^{L}\left(\hat{x}\right)\right\}\right)$. A diagonal argument gives that, up to a sub-sequence in $k$, for any $L\in\mathbb{N}$, $$\chi^{L}\cdot\hat{\nu}_{k}^{+}\xrightarrow{\text{weak $\ast$}} \beta^{L}\cdot\hat{\nu}^{L},\quad k\to+\infty$$ for some probability measure $\hat{\nu}^{L}$  where $\beta^{L}:=\lim_{k\to+\infty}\beta^{L}_{k}$. We note that, in general, since $\chi^{L}$ may not be a continuous function, even though $\hat{\nu}_{k}^{+} \to \hat{\mu}$, it is possible that $\beta^{L} \cdot \hat{\nu}^{L}$ does not coincide with $\chi^{L} \cdot \hat{\mu}$. Up to a sub-sequence in $L$, we assume that $$\beta^{L}\cdot\hat{\nu}^{L}\xrightarrow{\text{weak $\ast$}}\beta\cdot\hat{\mu}_{1},\quad L\to+\infty$$ for some probability measure $\hat{\mu}_{1}$ (whose projection is denoted by $\mu_{1}$) where $\beta:=\lim_{L\to+\infty}\beta^{L}$ (note that this is an increasing convergence). In the degenerate case $\beta=0$, the measure $\hat{\mu}_{1}$ plays no role in the subsequent argument. Therefore, we may define $\hat{\mu}_{1}$ to be the unstable lift of any ergodic measure $\mu_{1}$ whose positive Lyapunov exponent is larger than $\delta$.
	\begin{Claim}We have the following properties.
		\begin{itemize} 
			\item $\hat{\mu}_{1}$ is $\hat{f}$-invariant and consequently the probability measure $\hat{\mu}_{0}$ defined by $(1-\beta)\hat{\mu}_{0}=\hat{\mu}-\beta\hat{\mu}_{1}$ is also $\hat{f}$-invariant.
			\item $\lambda^{+}(f,x)\geq \delta$ for $\mu_{1}$-a.e. $x$ and $\hat{\mu}_{1}=\hat{\mu}_{1}^{+}$.
		\end{itemize}
	\end{Claim}
	\begin{proof}[Proof of Claim]
		Note that for the degenerate case $\beta = 0$, the claim is trivially true. We assume $\beta>0$ in the following argument.
		
		We first note that by definition, both $\hat{\mu}_{1}$ and $\hat{\mu}_{0}$ are positive measures since $0\leq\chi^{L}\cdot\hat{\nu}_{k}^{+}\leq \hat{\nu}_{k}^{+}$. Given $\hat{x}$, we decompose $E^{L}(\hat{x})\cap[0,+\infty)$ into maximal disjoint intervals: $$E^{L}(\hat{x})\cap[0,+\infty)=\bigcup_{i\geq 1}[a_{i}, a_{i}+b_{i}).$$ Given any continuous function $h:\hat{M}\to\mathbb{R}$, by the ergodicity of $\hat{\nu}_{k}^{+}$ (see Remark \ref{remark of entropy estimation of geometric component}), for $\nu_{k}^{+}$-a.e. $\hat{x}$,
		\begin{equation}\label{formula of m0}
			\left(\chi^{L}\cdot\hat{\nu}_{k}^{+}\right)(h)=\int\chi^{L}\cdot h\, d \hat{\nu}_{k}^{+}=\lim_{{j}\to+\infty}\frac1{{a}_{j}+{b}_{j}}\sum_{{i\leq j}} \left(S_{b_i}h\right)\left(\hat{f}^{a_i}(\hat{x})\right)
		\end{equation} where $\left(S_{n}h\right)(\hat{x}):=\sum_{i=0}^{n-1}h(\hat{f}^{i}(\hat{x}))$. Note that there are relatively few (at most $\frac{n}{L}+10$) such intervals $[a_{i}, a_{i}+b_{i})$, as the complement of $E^{L}_{n}(\hat{x})$ in $[0, n)$ consists of a disjoint union of intervals with lengths exceeding $L$ (except possibly the first and the last intervals). By (\ref{formula of m0}), we then have $$\left|\left(\chi^{L}\cdot\hat{\nu}_{k}^{+}\right)(h)-\left(\chi^{ L}\cdot\hat{\nu}_{k}^{+}\right)(h\circ\hat{f})\right|\leq \frac{2}{L}\cdot\max_{\hat{y}}|h(\hat{y})|.$$ Passing to the limit in $k$, we have $$\left|\left(\beta^{L}\cdot\hat{\nu}^{L}\right)(h)-\left(\beta^{L}\cdot\hat{\nu}^{L}\right)(h\circ\hat{f})\right|\leq \frac{2}{L}\cdot\max_{\hat{y}}|h(\hat{y})|.$$ Then passing to the limit in $L$ gives the invariance of $\hat{\mu}_{1}$. 
		
		Next we prove the second property. Note that by definition, for any point $\hat{x}$ and any integer $i\in E^{L}(\hat{x})$ (assuming $E^{L}(\hat{x})\neq \emptyset$ which is the case we are concerned with), there is $1\leq l\leq L$ such that $\hat{f}^{l}\left(\hat{f}^{i}\left(\hat{x}\right)\right)\in H_{\delta}$. By the definition of $H_{\delta}$, we have (recalling $\rho(x, E):=\log\|Df_{x}|_{E}\|$) $$\hat{f}^{i}(\hat{x})\in \Omega^{L}:=\left\{\hat{y}:~ \exists 1\leq l\leq L \text{ such that } \sum_{m=0}^{l-1}\rho\left(\hat{f}^{m}\left(\hat{y}\right)\right)\geq l\cdot\delta\right\}.$$ Note that in the definition of $\Omega^{L}$ above, we use the larger function $\rho$ rather than $\rho'$ (which is used in the definition of $H_{\delta}$). As a consequence, $$\lim_{n\to+\infty}\Big(\frac{1}{n}\sum_{i\in E^{L}_{n}(\hat{x})}\delta_{\hat{f}^{i}(\hat{x})}\Big)\left(\Omega^{L}\right)=\lim_{n\to+\infty}\Big(\frac{1}{n}\sum_{i\in E^{L}_{n}(\hat{x})}\delta_{\hat{f}^{i}(\hat{x})}\Big)\left(\hat{M}\right).$$ Assuming that $\hat{x}$ satisfies (\ref{weak star convergence for typical points}) (such points carry $\hat{\nu}_{k}^{+}$-full measure), by the compactness of $\Omega^{L}$, the quantities above coincide with $$\left(\chi^{L}\cdot\hat{\nu}_{k}^{+}\right)\left(\Omega^{L}\right)=\left(\chi^{L}\cdot\hat{\nu}_{k}^{+}\right)\left(\hat{M}\right).$$ Again by the compactness of $\Omega^{L}$, passing to the limit in $k$, $$\left(\beta^{L}\cdot\hat{\nu}^{L}\right)\left(\Omega^{L}\right)=\left(\beta^{L}\cdot\hat{\nu}^{L}\right)\left(\hat{M}\right)=\limsup_{k\to+\infty}\left(\chi^{L}\cdot\hat{\nu}_{k}^{+}\right)\left(\Omega^{L}\right)=\limsup_{k\to+\infty}\left(\chi^{L}\cdot\hat{\nu}_{k}^{+}\right)\left(\hat{M}\right)=\beta^{L}$$ and consequently, since $\beta^{L}\cdot\hat{\nu}^{L}$ converges increasingly  \footnote{Since $\chi^{L}\leq \chi^{L+1}$, the convergence $\beta^{L}\cdot\hat{\nu}^{L}\to\beta\cdot\hat{\mu}_{1}$ w.r.t. $L$ is an increasing convergence where \emph{increasing} means for any subset $A$, $$\beta^{L}\cdot\hat{\nu}^{L}(A)\leq \beta^{L+1}\cdot\hat{\nu}^{L+1}(A).$$} to $\beta\cdot\hat{\mu}_{1}$ as $L\to+\infty$, $$\left(\beta\cdot\hat{\mu}_{1}\right)\left(\bigcup_{L}\Omega^{L}\right)=\left(\beta\cdot\hat{\mu}_{1}\right)\left(\hat{M}\right)=\beta.$$ Therefore the $\hat{f}$-invariant set $\bigcap_{i\in\mathbb{Z}}\hat{f}^{i}\left(\bigcup_{L}\Omega^{L}\right)$ has $\hat{\mu}_{1}$ full measure, and for any $\hat{x}$ in this set, by definition, there is an increasing sequence $0=m_{0}<m_{1}<m_{2}<\cdots$ such that $$\sum_{m=m_{i}}^{m_{i+1}-1}\rho\left(\hat{f}^{m}\left(\hat{x}\right)\right)\geq \left(m_{i+1}-m_{i}\right)\cdot\delta.$$ Hence $$\limsup_{n\to+\infty}\frac{1}{n}\sum_{i=0}^{n-1}\rho\left(\hat{f}^{i}\left(\hat{x}\right)\right)\geq \delta.$$ Since we have assumed that  the  projection $\mu$ of $\hat{\mu}$ has only one positive Lyapunov exponent, $\hat{\mu}_{1}$ (as a component of $\hat{\mu}$) must coincide with the (unique) unstable lift $\hat{\mu}_{1}^{+}$ of its projection $\mu_{1}$.

	\end{proof}

	Given a sequence of numbers $\varepsilon_{k}\to 0$, let $\delta_{\nu_{k}}>0$ be sufficiently small such that $\nu_{k}(\Lambda'_{k})>\frac{3}{4}$ where  $$\Lambda'_{k}:=\left\{x\in \mathcal{R}:~h(f,\nu_{k})-\varepsilon_{k}\leq\liminf_{n\to+\infty}-\frac{1}{n}\log\left(\nu_{k}\right)_{x}\left(B_{n}\left(x,\delta_{\nu_{k}}\right)\right)\right\}.$$ Let  $\Lambda''_{k}$ be another subset with $\nu_{k}(\Lambda''_{k})>\frac{3}{4}$ such that for any $L\in\mathbb{N}\cup\{\infty\}$, we have with $\hat{x}:=(x, E^{+}_{x})$, $$\Big(\frac{1}{n}\sum_{i\in E^{L}_{n}(\hat{x})}\delta_{\hat{f}^{i}(\hat{x})}\Big)\xrightarrow{\text{weak $\ast$}} \chi^{L}\cdot\hat{\nu}_{k}^{+},\quad \text{uniformly in $x\in\Lambda''_{k}$}.$$  Define $$\Lambda_{k}:=\Lambda'_{k}\cap\Lambda''_{k}$$ and note $\nu_{k}(\Lambda_{k})>\frac{1}{2}$.

	Given any $q\in\mathbb{N}$,  consider any finite partitions $\mathcal{P}_{\nu_{k}}, \mathcal{Q}$ on $M$ with the following properties:
	\begin{itemize}
		\item $\diam(\mathcal{P}_{\nu_{k}})\leq \delta_{\nu_{k}}, \diam(\mathcal{Q})\leq \delta_{q},$
		\item $\nu_{k}(\partial\mathcal{P}_{\nu_{k}})=0$,
		\item for any $m\in\mathbb{N}$, $\nu_{k}(\partial\mathcal{Q}^{m})=0$, and consequently, $\hat{\nu}^{+}_{k}\left(\partial\left(\pi^{-1}\mathcal{Q}^{m}\right)\right)=0,$
		\item for any $L, m\in\mathbb{N}$,  $\hat{\nu}^{L}\left(\partial\left(\pi^{-1}\mathcal{Q}^{m}\right)\right)=0,$
		\item for any $m\in\mathbb{N}$, $\mu_{1}(\partial\mathcal{Q}^{m})=0$, and consequently, $\hat{\mu}^{+}_{1}\left(\partial\left(\pi^{-1}\mathcal{Q}^{m}\right)\right)=0.$
	\end{itemize}

	By Proposition \ref{entropy splitting}, Proposition \ref{entropy estimation of geometric component} and Proposition \ref{entropy estimation of neutral component}, for $\nu_{k}$-a.e. $x\in\Lambda_{k}$ and $L, m\in\mathbb{N}$, 
	\begin{equation*}
		\begin{aligned}
			h(f,\nu_{k})-\varepsilon_{k}
			&\leq\limsup_{n\to+\infty}\frac{1}{n}H_{\left(\nu_{k}\right)_{x,\Lambda_{k}}}\left(\mathcal{Q}^{\mathcal{E}_{n}^{L}}\bigg|\mathcal{E}_{n}^{L}\right)+ \limsup_{n\to+\infty}\frac{1}{n}H_{\left(\nu_{k}\right)_{x,\Lambda_{k}}}\left(\mathcal{P}_{\nu_{k}}^{[0,n)}\bigg|\mathcal{Q}^{\mathcal{E}_{n}^{L}}\right)+H\left(\frac{2}{L}\right)\\
			&\leq\beta^{L}_{k}\frac{1}{m}H_{(\beta_{k}^{L})^{-1}\chi^{L}\hat{\nu}_{k}^{+}}\left(\pi^{-1}\mathcal{Q}^{m}\right)+\frac{12m\log\#\mathcal{Q}}{L}\\
			&+\left(1-\beta_{k}^{L}\right)\cdot C(f)+\left(\log 2+\frac{1}{r-1}\right)\cdot \left(\int\frac{\log^{+}\left\|Df^{q}_{x}\right\|}{q}d\left(\pi_{*}\left(\chi^{L}\hat{\nu}_{k}^{+}\right)\right)\left(x\right)-\int\rho d\left(\chi^{L}\hat{\nu}_{k}^{+}\right)\right)\\
			&+ H\left(\frac{2}{L}\right)+\tau_{1}(L,q)
		\end{aligned}
	\end{equation*} where the error term $\tau_{1}(L,q)$ satisfies $$\limsup_{q\to+\infty}\limsup_{L\to+\infty}\tau_{1}(L,q)=0.$$ Note that for the degenerate case $\beta_{k}^{L}=0$, by Proposition \ref{entropy estimation of geometric component}, we can just ignore the term $$\beta^{L}_{k}\frac{1}{m}H_{(\beta_{k}^{L})^{-1}\chi^{L}\hat{\nu}_{k}^{+}}\left(\pi^{-1}\mathcal{Q}^{m}\right)+\frac{12m\log\#\mathcal{Q}}{L}.$$

	Passing to the limits in $k$ and then in $L$, we have 
	\begin{equation*}
		\begin{aligned}
			\limsup_{k\to+\infty}h(f,\nu_{k})
			&\leq\beta\frac{1}{m}H_{\hat{\mu}_{1}}\left(\pi^{-1}\mathcal{Q}^{m}\right)\\
			&+(1-\beta)\cdot C(f)+\left(\log 2+\frac{1}{r-1}\right)\cdot \left(\int\frac{\log^{+}\left\|Df^{q}_{x}\right\|}{q}d\left(\beta\mu_{1}\right)\left(x\right)-\int\rho d\left(\beta\hat{\mu}_{1}\right)\right)\\
			&+\limsup_{L\to+\infty}\tau_{1}(L,q).
		\end{aligned}
	\end{equation*}
	Since $\hat{\mu}_{1}=\hat{\mu}_{1}^{+}$, $$\int\frac{\log^{+}\left\|Df^{q}_{x}\right\|}{q}d\mu_{1}\left(x\right)\to\lambda^{+}(f,\mu_{1})=\int\rho d\hat{\mu}_{1}, \quad q\to+\infty.$$ Hence by noting $h(\hat{f},\hat{\mu}_{1})=h(f,\mu_{1})$ (see Section \ref{tangent dynamics}) and  letting $q,m \to +\infty$, we have $$\limsup_{k\to+\infty}h(f,\nu_{k})\leq\beta h(\hat{f},\hat{\mu}_{1})+(1-\beta)\cdot C(f)=\beta h(f,\mu_{1})+(1-\beta)\cdot C(f).$$
\end{proof}
Now we turn to the proof of Theorem \ref{invariant Cr case}. The idea is to apply Proposition \ref{invariant cr case with error term} to large iterates of $f$.
\begin{proof}[Proof of Theorem \ref{invariant Cr case}]
	Note that by definition, $$\lim_{p\to+\infty}\frac{C(f^{p})}{p}=\frac{\lambda^{+}(f)}{r}.$$ Hence given $\alpha>\frac{\lambda^{+}(f)}{r}$, we can choose $p\in\mathbb{N}$ large enough such that $$\alpha>\frac{C(f^{p})}{p}.$$ Given $k$, consider any ergodic component $\nu_{k}^{p}$ of $\nu_{k}$ for $f^{p}$ and the corresponding unstable lift $\left(\hat{\nu}_{k}^{p}\right)^{+}$ for $\hat{f}^{p}$. Up to a sub-sequence, we may assume that $\left(\hat{\nu}_{k}^{p}\right)^{+}$ converges to some $\hat{f}^{p}$-invariant measure $\hat{\mu}^{p}$. Note that the projection $\mu_{p}$ has exactly one positive Lyapunov exponent w.r.t. $f^{p}$, since $\mu=\frac{1}{p}\sum_{i=0}^{p-1}f^{i}_{*}\mu_{p}$ has exactly one positive Lyapunov exponent by our assumption.

	We may assume $\lim_{k\to+\infty} \lambda^{+}(f^{p},\nu^{p}_{k})>\frac{\log\|Df^{p}\|}{r}$, otherwise set $\beta=0, \hat{\mu}_{0}=\hat{\mu}$ and the inequality in Theorem \ref{invariant Cr case} is trivially true by the following estimation using Ruelle's inequality: $$\limsup_{k\to+\infty}h(f,\nu_{k})=\limsup_{k\to+\infty}\frac{h(f^{p},\nu^{p}_{k})}{p}\leq \limsup_{k\to+\infty}\frac{\lambda^{+}(f^{p},\nu^{p}_{k})}{p}\leq\frac{\log\left\|Df^{p}\right\|}{pr}\leq\frac{C(f^{p})}{p}<\alpha.$$
	
	Applying Proposition \ref{invariant cr case with error term} to the system $f^{p}$ with the measures $\left(\hat{\nu}_{k}^{p}\right)^{+}$ and $ \hat{\mu}^{p}$, we get a decomposition of $\hat{\mu}^{p}=(1-\beta)\hat{\mu}_{0}^{p}+\beta (\hat{\mu}^{p}_{1})^{+}$ for some $\beta\in [0,1]$ and some $\hat{f}^{p}$-invariant measures $\hat{\mu}_{0}^{p},(\hat{\mu}_{1}^{p})^{+}$ ( $(\hat{\mu}_{1}^{p})^{+}$ is the unstable lift of some $f^{p}$-invariant measure $\mu_{1}^{p}$) such that 
	\begin{equation}\label{final result for f p}
		\limsup_{k\to+\infty}h(f^{p},\nu^{p}_{k})\leq\beta h(f^{p},\mu_{1}^{p})+(1-\beta)\cdot C(f^{p}).
	\end{equation} Define $$\hat{\mu}_{1}^+:=\frac{1}{p}\sum_{i=0}^{p-1}\hat{f}^{i}_{*}\left(\hat{\mu}_{1}^{p}\right)^{+}, \quad \hat{\mu}_{0}:=\frac{1}{p}\sum_{i=0}^{p-1}\hat{f}^{i}_{*}\hat{\mu}_{0}^{p}.$$ Let $\mu_{1}, \mu_{0}$ be the projections of $\hat{\mu}_{1}^{+}, \hat{\mu}_{0}$ and note that $\hat{\mu}_{1}^{+}$ is the unstable lift of $\mu_{1}$. Also note $$\hat{\mu}=\frac{1}{p}\sum_{i=0}^{p-1}\hat{f}^{i}_{*}\hat{\mu}^{p}$$ and consequently, we have a decomposition of $$\hat{\mu}=(1-\beta)\hat{\mu}_{0}+\beta \hat{\mu}^{+}_{1}.$$ Hence by (\ref{final result for f p}) and the formula $h(f^{p},\mu_{1}^{p})=p\cdot h(f,\mu_{1})$, $$\limsup_{k\to+\infty}h(f,\nu_{k})=\limsup_{k\to+\infty}\frac{h(f^{p},\nu^{p}_{k})}{p}\leq\beta h(f,\mu_{1})+(1-\beta)\frac{C(f^{p})}{p}\leq\beta h(f,\mu_{1})+(1-\beta)\alpha.$$

\end{proof}

\section{Appendix}

Given a subset $\mathbb{L}\subset \mathbb{N}$, we define its \emph{upper density} by $$\overline{d}(\mathbb{L}):=\limsup_{n\to+\infty}\frac{\#\left([0,n)\cap \mathbb{L}\right)}{n}.$$
\begin{Lemma}[Pliss Lemma, {\cite[Lemma 4.12 ]{MiZ18}}]\label{Pliss Lemma}
Given $\gamma_{1}<\gamma_{2}<C$, let $\beta:=\frac{\gamma_{2}-\gamma_{1}}{C-\gamma_{1}}$.  If a sequence $\{a_{n}\}$ of numbers satisfies 
\begin{itemize}
	\item $\lim_{n\to+\infty}\frac{1}{n}\sum_{i=0}^{n-1}a_{i}\geq\gamma_{2}$,
	\item $a_{n}\leq C$ for any $n\in\mathbb{N}$,
\end{itemize}
then there is a subset $\mathbb{L}\subset \mathbb{N}$ with $\overline{d}(\mathbb{L})\geq\beta$ such that for any $j\in\mathbb{L}$ and any $n\in\mathbb{N}$, $$\frac{1}{n}\sum_{i=0}^{n-1}a_{j+i}\geq\gamma_{1}.$$
\end{Lemma}

The following consequence of Pliss Lemma is well known. We provide a proof for completeness.

\begin{Lemma}\label{Lem:abstract-ergodic}
	Assume that $(X,T)$ is a dynamical system. For any $\gamma\in(0,1)$, any ergodic measure $\nu$ and any measurable set $U\subset X$, if $\nu(U)\geq 1-\gamma^{2}$, then
	$\nu(\Omega)\ge 1-\gamma$ where $$	\Omega:=\left\{x\in X:~\forall \ell\ge1,~\frac{1}{\ell}\#\left\{k:~0\le k<\ell,~T^k(x)\in U\right\}\geq 1-\gamma\right\}.$$
\end{Lemma}
\begin{proof}
By Birkhoff's ergodic theorem, for $\nu$-a.e. $y$, $$\lim_{n\to+\infty}\frac{1}{n}\sum_{i=0}^{n-1}\chi_{U}\circ T^{i}(y)=\nu(U)\geq1-\gamma^{2}$$ where $\chi_{U}$ denotes the indicator function of $U$. Applying Lemma \ref{Pliss Lemma} with $$\gamma_{1}=1-\gamma,\quad \gamma_{2}=1-\gamma^{2},\quad C=1, \quad a_{n}=\chi_{U}\circ T^{n}(y),$$  we have $\beta=1-\gamma$ and a subset $\mathbb{L}(y)\subset \mathbb{N}$ with $\overline{d}(\mathbb{L}(y))\geq\beta=1-\gamma$ such that for any $j\in\mathbb{L}(y)$, $$T^{j}(y)\in \Omega.$$ Again, by Birkhoff's ergodic theorem, we have $\nu(\Omega)\ge 1-\gamma$.
\end{proof}

\begin{Lemma}[{\cite[Lemma III.8]{Man83}}]\label{Mane Temper function}
	Let $(X,f)$ be a dynamical system and let $A:X\to (0,+\infty)$ be a measurable function. If there is a constant $K>1$ such that for any $x\in X$, $$K^{-1}\leq\frac{A(f(x))}{A(x)}\leq K,$$ then for any $x$ in a total measure set (i.e., a subset with full measure for any invariant measure), $$\lim_{n\to+\infty} \frac{1}{n}\log A(f^{n}(x))=0.$$

\end{Lemma}
\begin{Lemma}\label{angle uniformly bounded from below}
		Let $f$ be a $C^{1}$ diffeomorphism on a compact manifold $M$ without boundary. For any $N\in\mathbb{N}$, there is a constant $C>1$ such that for any $x\in M$ and any two unit vectors $v,w\in T_{x}M$, if 
		\begin{equation}\label{assumption in lemma angle uniformly bounded from below in Appendix}
			\left\|Df^{N}_{x}(v)\right\|> 2,\quad \left\|Df^{N}_{x}(w)\right\|<1,
		\end{equation} then $$\measuredangle\left(v,w\right)\geq C^{-1}.$$	
\end{Lemma}
\begin{proof}
Given $N\in\mathbb{N}$, by continuity, there is $C>1$ such that for any $x\in M$ and any two unit vectors $v,w\in T_{x}M$, if $\measuredangle(v,w)< C^{-1}$, then $$\frac{\left\|Df^{N}_{x}(v)\right\|}{\left\|Df^{N}_{x}(w)\right\|}\leq 2.$$ If two unit vectors $v,w$ satisfy condition (\ref{assumption in lemma angle uniformly bounded from below in Appendix}), then we must have $\measuredangle(v,w)\geq C^{-1}$, otherwise the above inequality would imply  $$\left\|Df^{N}_{x}(v)\right\|\leq 2 \left\|Df^{N}_{x}(w)\right\|<2$$ which leads to a contradiction.
\end{proof}

\begin{Lemma}\label{ergodic decomposition of flow}
Let $\varphi=\{\varphi^t\}_{t\in\mathbb{R}}$ be a continuous flow on a compact metric space $X$ and let $\mu$ be an ergodic measure of $\varphi$. Then for (almost) every ergodic component $\nu$ of $\mu$ w.r.t. the time-one map $\varphi^{1}$, we have $$\mu=\int_{0}^{1}\varphi^{t}_{*}\left(\nu\right) dt$$ where $\varphi^{t}_{*}\left(\nu\right)$ is the  pushforward measure of $\nu$, i.e., $\varphi^{t}_{*}\left(\nu\right)\left(A\right):=\nu\left(\varphi^{-t}\left(A\right)\right).$
\end{Lemma}
\begin{proof}
	Let $\Lambda$ denote the flow-invariant set of $x$ such that $$\frac{1}{T}\int_{0}^{T}\delta_{\varphi^t\left(x\right)}dt\xrightarrow{\text{weak $\ast$}}\mu.$$ By Birkhoff's ergodic theorem, $\mu(\Lambda)=1$. Let $\nu$ be an ergodic component of $\mu$ w.r.t. the time-one map $\varphi^{1}$ such that $\nu(\Lambda)=1$ (note this holds for almost every ergodic component). By definition, one can check that $\int_{0}^{1}\varphi^{t}_{*}\left(\nu\right) dt$ is an invariant measure w.r.t. the flow. Given an invariant subset $A$ (i.e., $\varphi^{t}(A)=A, \forall\,t\in\mathbb{R}$), by the ergodicity of $\nu$ w.r.t. $\varphi^{1}$, we have either $\nu(A)=0$ or $\nu(A)=1$ which implies either $\varphi^{t}_{*}\left(\nu\right)\left(A\right)=0$ or $\varphi^{t}_{*}\left(\nu\right)\left(A\right)=1$. This proves ergodicity of $\int_{0}^{1}\varphi^{t}_{*}\left(\nu\right) dt$ w.r.t. the flow and consequently, it must coincide with $\mu$, since $\int_{0}^{1}\varphi^{t}_{*}\left(\nu\right) dt$ assigns full measure to $\Lambda$.
\end{proof}
Recall $$H(t):=t\log \frac 1 t+(1-t)\log\frac 1{1-t},\,t\in(0,1).$$
\begin{Lemma}\label{lemma:Stirling's formula}
	Let $0<\alpha<\frac12$ and let $k_n$ be a sequence of integers satisfying
	$$\frac{k_n}{n}\to \alpha,\quad n\to +\infty.$$
	Then
	$$\lim_{n\to+\infty}\frac1n\log\binom{n}{k_n}=H(\alpha).$$
	Moreover,
	$$\lim_{n\to+\infty}\frac1n\log\left(\sum_{m=0}^{k_n}\binom{n}{m}\right)=H(\alpha).$$
\end{Lemma}

\begin{proof}
	Recall Stirling's formula: $$n! \sim \sqrt{2\pi n}\left(\frac{n}{e}\right)^n\qquad (n\to+\infty).$$ As a consequence, $$\log(n!) = n\log n-n+o(n).$$
	Hence
	$$
	\begin{aligned}
		\log\binom{n}{k_n}
		&=\log(n!)-\log(k_n!)-\log((n-k_n)!)\\
		&=n\log n-k_n\log k_n-(n-k_n)\log(n-k_n)+o(n).
	\end{aligned}
	$$
	Since $k_n=\alpha n+o(n)$, write $k_n=\alpha n+\delta_n$ with $\delta_n=o(n)$. Then
	$$
	\begin{aligned}
		k_n\log k_n &= (\alpha n+\delta_n)\bigl[\log n+\log(\alpha+\delta_n/n)\bigr]\\
		&= \alpha n\log n+\delta_n\log n+\alpha n\log(\alpha+\delta_n/n)+o(n),\\[4pt]
		(n-k_n)\log(n-k_n) &= ((1-\alpha)n-\delta_n)\bigl[\log n+\log(1-\alpha-\delta_n/n)\bigr]\\
		&= (1-\alpha)n\log n-\delta_n\log n+(1-\alpha)n\log(1-\alpha-\delta_n/n)+o(n).
	\end{aligned}
	$$
	Adding these, the $\pm\delta_n\log n$ terms cancel, and substituting back yields
	$$
	\begin{aligned}
		\log\binom{n}{k_n}
		&= -\alpha n\log(\alpha+\delta_n/n)-(1-\alpha)n\log(1-\alpha-\delta_n/n)+o(n)\\
		&= -\alpha n\bigl(\log\alpha+o(1)\bigr)-(1-\alpha)n\bigl(\log(1-\alpha)+o(1)\bigr)+o(n)\\
		&= n\bigl(-\alpha\log\alpha-(1-\alpha)\log(1-\alpha)\bigr)+o(n).
	\end{aligned}
	$$ This gives
	$$\lim_{n\to+\infty}\frac1n\log\binom{n}{k_n}=H(\alpha).$$
	
	Since $\alpha<\frac12$, for sufficiently large $n$, the sequence
	$$\binom{n}{0},\binom{n}{1},\ldots,\binom{n}{k_n}$$
	is increasing. Therefore
	$$\binom{n}{k_n}\le\sum_{m=0}^{k_n}\binom{n}{m}\le (k_n+1)\binom{n}{k_n}.$$
	Taking logarithms, dividing by $n$, and using
	$$\frac{\log(k_n+1)}{n}\to 0,$$
	the conclusion follows from the first part.
\end{proof}

\hspace*{\fill}
\begin{Acknowledgements}
	
	We would like to express our deepest gratitude to J\'{e}r\^{o}me Buzzi from Universit\'{e} Paris-Saclay for early sharing ideas from \cite{BCS25}. We sincerely thank Dawei Yang from Soochow University for his constructive suggestions and David Burguet from Sorbonne Universit\'{e} for introducing us to his work on the continuity of Lyapunov exponents during his visit to Soochow University. We also appreciate Chiyi Luo, Zeya Mi and Rui Zou for their valuable discussions.
	
	We are very grateful to the anonymous reviewers for their careful and thoughtful reports. Their comments not only helped us correct a number of inaccuracies and typos, but also led to several improvements in the exposition and organization of the paper.
	
\end{Acknowledgements}

\vskip 5pt

\flushleft{\bf Yuntao Zang} \\
\small School of Mathematical Sciences, Soochow University, Suzhou, 215006, P.R. China\\
\textit{E-mail:} \texttt{ytzang@suda.edu.cn}\\

\end{document}